\documentclass[11pt]{article}
\usepackage[a4paper, margin = 2.2cm]{geometry}

\usepackage[english]{babel}
\usepackage{latexsym,amsmath,enumerate,graphics,enumerate,amsthm,hyperref,float}
\usepackage[mathscr]{euscript}
\usepackage[affil-it]{authblk}
\usepackage{enumerate,amsthm,dsfont,pstricks}
\usepackage{latexsym,amsmath,amssymb,amscd,wrapfig,graphicx}
\usepackage{color}

\newtheorem{theorem}{Theorem}[section]
\newtheorem{lemma}[theorem]{Lemma}
\newtheorem{proposition}[theorem]{Proposition}
\newtheorem{corollary}[theorem]{Corollary}

\theoremstyle{remark}
\newtheorem*{remark*}{\bf Remark}

\theoremstyle{definition}
\newtheorem{definition}[theorem]{Definition}

\newtheorem{remark}[theorem]{Remark}

\def\sqr#1#2{{\,\vcenter{\vbox{\hrule height.#2pt\hbox{\vrule width.#2pt
height#1pt \kern#1pt\vrule width.#2pt}\hrule height.#2pt}}\,}}
\def\bo{\sqr44\,}
\def\q{\quad}

\let\phi=\varphi
\def\epsilon{\varepsilon}

\def\R{\mathbb{R}}

\def\0{\mathbf{0}}
\def\to{\longrightarrow}

\def\re{\mathrm{Re}}
\def\JB{\mathrm{JB}}

\def\var{\mathrm{var}}

\def\ol{\overline}

\DeclareMathOperator{\Exp}{\mathrm{exp}_0}

\newcommand{\comment}[1]{}

\numberwithin{equation}{section}

\usepackage{sectsty}
\sectionfont{\centering}

\let\epsilon=\varepsilon

\makeatletter
\def\@maketitle{%
  \newpage
  \null
  \vskip 2em%
  \begin{center}%
  \let \footnote \thanks
    {\Large\bfseries \@title \par}%
    \vskip 1.5em%
    {\normalsize
      \lineskip .5em%
      \begin{tabular}[t]{c}%
        \@author
      \end{tabular}\par}%
    \vskip 1em%
    {\normalsize \@date}%
  \end{center}%
  \par
  \vskip 1.5em}
\makeatother

\begin{document}

\title{\LARGE Horofunctions and metric compactification
of\\ noncompact Hermitian symmetric spaces\footnote{The authors gratefully acknowledge the support by the EPSRC (UK) (grant EP/R044228/1)}
}

\author[1]{Cho-Ho Chu%
\thanks{Email: \texttt{c.chu@qmul.ac.uk}}}
\affil[1]{School of Mathematical Sciences, Queen Mary, University of London, \newline London E1 4NS, UK}

\author[2]{Mar\'ia Cueto-Avellaneda%
\thanks{Email: \texttt{M.Cueto-Avellaneda@kent.ac.uk}}}
\affil[2]{School of Mathematics, Statistics \& Actuarial Science,
University of Kent,\newline Canterbury CT2 7NX, UK}

\author[2]{Bas Lemmens%
\thanks{Email: \texttt{B.Lemmens@kent.ac.uk}}}
\date{}

\maketitle


\begin{abstract}
Given a Hermitian symmetric space $M$ of noncompact type,  we give a complete description of the horofunctions in the metric compactification of $M$ with respect to the Carath\'eodory distance, via the realisation of  $M$ as the open unit
ball $D$ of a Banach space $(V,\|\cdot\|)$ equipped with a Jordan structure, called a $\mathrm{JB}^*$-triple.
The Carath\'eodory distance $\rho$ on $D$   has a Finsler structure. It is the integrated distance of
the Carath\'eodory differential metric, and the norm $\|\cdot\|$  in the realisation is the Carath\'eodory norm with respect to the origin $0\in D$.  We also identify the horofunctions of the metric compactification of $(V,\|\cdot\|)$  and relate  its geometry and global topology to the closed dual unit ball (i.e., the polar of $D$).  Moreover, we show that the exponential map $\exp_0 \colon V \longrightarrow D$ at
$0\in D$ extends to a homeomorphism between the metric compactifications of $(V,\|\cdot\|)$  and $(D,\rho)$, preserving the geometric structure. Consequently, the metric compactification of $M$ admits a concrete realisation as the  closed dual unit ball of $(V,\|\cdot\|)$.
\end{abstract}




\tableofcontents
\sectionfont{}
\section{Introduction}

Compactifications of symmetric spaces is a particularly rich subject, which has been studied extensively \cite{bj,gjt}.
A variety  of compactifications of symmetric spaces have been introduced with different applications in mind. For instance, Satake \cite{Sa} introduced his compactifications in his study of automorphic forms, and the Martin and Furstenberg compactifications \cite{Fu, Ma} were introduced to analyse harmonic functions.

Recently, metric (or horofunction) compactifications with respect to invariant Finsler metrics have been used to investigate different types of compactifications of symmetric spaces. In particular,  it was shown in \cite{HSWW,S1} that  generalised Satake compactifications and  Martin compactifications of  symmetric spaces can be realised as  metric compactifications under   suitable invariant Finsler metrics. In \cite{KL} an explicit invariant Finsler metric was constructed on symmetric spaces whose metric compactification gives the maximal Satake compactification, and in \cite{GJ} the minimal Satake compactification of $\mathrm{SL}_n(\mathbb{R})/\mathrm{SO}_n$ was realised as a metric compactification.

However, for relatively few noncompact type symmetric spaces, with invariant Finsler metrics, is the metric compactification well understood.  In this paper we completely determine the metric compactifcation of noncompact type Hermitian symmetric spaces under the Carath\'eodory distance and provide a detailed analysis of its geometry and global topology.  The  Carath\'eodory distance plays an important role in the geometry and analysis of Hermitian symmetric spaces, hence it is natural to consider its metric compactification.

Given a Hermitian symmetric space $M$ of noncompact type, the Harish-Chandra embedding,
 $$M\approx \mathcal{D} \subset \mathfrak p^+ \hookrightarrow M_c,$$
 identifies $M$ bihomorphically with a bounded symmetric domain $\mathcal{D}$ in a complex Euclidean space $\mathfrak p^+$,
 which is biholomorphic to an open dense subset of the compact dual $M_c$ of $M$. By the seminal works of Loos \cite{loos} and Kaup \cite{kaup}, one can equip $\mathfrak p^+$ with a Jordan algebraic structure and a norm $\|\cdot\|$ so that
$\mathcal{D}$  is biholomorphic to the open unit ball $D$ of  $(\mathfrak p^+, \|\cdot\|)$, which is called a {\em $\mathrm{JB}^*$-triple.}
Since the Carath\'eodory distance is invariant under biholomorphisms, one can transfer from $M$ to the open unit ball $D$ in the $\mathrm{JB}^*$-triple $V = (\mathfrak p^+, \|\cdot\|)$ to study the metric compactification and exploit Jordan theory and functional analysis. This is our task in the paper.

For symmetric spaces of noncompact type, the metric compactification with respect to the Riemannian distance can be identified with its geodesic compactification (\cite{ball} and \cite[Proposition 12.6]{jm}), and is known to be homeomorphic to a Euclidean ball, see  \cite[Chapter II.8]{BH}.  It has also been observed for various classes of finite dimensional normed spaces that the geometry and global topology of the metric compactification is closely related to the closed dual unit ball of the norm. In \cite{JS1,JS} this connection was established for  normed spaces with polyhedral unit balls, see also \cite{CKS}. At present, however, it is unknown \cite[Question 6.18]{KL} if this duality phenomenon holds for general finite dimensional normed spaces.

In \cite{L2,LP} it was shown that the duality phenomenon does not only appear in finite dimensional normed spaces, but also occurs in metric compactifications of certain symmetric spaces with invariant Finsler norms. More precisely, it was shown in \cite{L2,LP} for symmetric cones equipped with the Thompson  and  Hilbert distances that the geometry and global topology of the metric compactification coincides with the geometry of the closed dual unit ball of the Finsler norm in the tangent space at the base point of the metric compactification.

In these symmetric spaces $N$ the connection between the geometry of the metric compactification and the dual unit ball manifests itself in the following way. The horofunction boundary $N(\infty)$ in the metric compactification of the symmetric space $N$ carries an equivalence relation where two horofunctions $g$ and $h$ are equivalent if $\sup_{x\in N} |g(x)-h(x)|<\infty$. This relation yields a natural partition of the horofunction boundary.    On the other hand, the boundary of the closed dual unit ball  $B^*$ of the Finsler norm on the tangent space $T_bN$ at the base point $b\in N$, is partitioned by the relative interiors of its boundary faces.   For the symmetric spaces $N$ considered in \cite{L2,LP} it was shown that there exists a homeomorphism $\phi$ from the metric compactification $N\cup N(\infty)$ onto  $B^*$, which maps each equivalence class of $N(\infty)$ onto the relative interior of a boundary face of $B^*$. In this sense the dual ball captures the geometry of the metric compactification.

Moreover, in \cite{L2} it was observed that the metric compactification of $N$  is closely related to the metric compactification of the normed space  $(T_bN,\|\cdot\|_b)$, where $\|\cdot\|_b$ is the Finsler norm  on the tangent space $T_bN$  at the base point $b\in N$. More explicitly, it was shown that the exponential map $\exp_b\colon T_bN\to N$ extends to a homeomorphism between the metric compactifications of the normed space $(T_bN,\|\cdot\|_b)$  and  $N$ under the invariant Finsler distance.  Furthermore, the extension preserves the equivalence classes in the horofunction boundaries.   A case in point is the  Hilbert distance on  $\mathrm{SL}_n(\mathbb{C})/\mathrm{SU}_n$, which can be realised as the manifold of positive-definite Hermitian $n\times n$ matrices with determinant $1$.  The tangent space at the identity $I$ is the real vector space of the $n\times n$ Hermitian matrices with trace $0$ and has Finsler norm, $|A|_I = \max \sigma(A) -\min \sigma(A)$, where $\sigma(A)$ is the set of eigenvalues of $A$.

 For a noncompact type Hermitian symmetric space $M$  identified as the open unit ball  $D$ of a JB*-triple $V=(\frak p^+, \|\cdot\|)$, with Carath\'eodory distance $\rho$, the $\mathrm{JB}^*$-triple norm $\|\cdot\|$ is the corresponding Finsler norm
 \cite[(4.5)]{kaup}. In this paper we show that the analogues of the results in \cite{L2,LP}  hold for the metric compactification of $(D, \rho)$.

More specifically, the following results are established. We provide a complete description of the horofunctions of $(D, \rho)$  in Theorem  \ref{hbsd}. We also determine the horofunctions of the $\mathrm{JB}^*$-triple $(V,\|\cdot\|)$ in Theorem \ref{formofh}, and establish in
 Theorem \ref{homeo}
 an explicit  homeomorphism $\phi$ between the horofunction compactification of $(V,\|\cdot\|)$ and its closed dual unit ball. Further, we show in Section \ref{GeomComp} that the homeomorphism $\phi$ maps each equivalence class in the horofunction boundary of $(V,\|\cdot\|)$ onto the relative interior of a boundary face of the closed dual unit ball.  Finally we prove in Section \ref{ExpHom} that the exponential map $\Exp\colon V\to D$  extends to  a homeomorphism between the metric compactifications $V \cup V(\infty)$ and $D \cup D(\infty)$, which maps equivalence classes onto equivalence classes in the horofunction boundaries, see Theorem \ref{exp}. Combining the results we see that the geometry and global topology of the metric compactification of $(D, \rho)$ coincides with the geometry of the closed dual unit ball of $(V,\|\cdot\|)$.

We start by recalling the essential background on metric compactifications and the theory of $\mathrm{JB}^*$-triples in Sections \ref{Hor} and \ref{jordan}, respectively.

\section{Horofunctions}\label{Hor}

The origins of the idea of the metric (or horofunction) compactification  go back to Gromov \cite{ball}. It has proven to be a valuable tool in numerous fields such as, geometric group theory, complex and real dynamics, Riemannian geometry, and geometric analysis.  It captures asymptotic geometric properties of metric spaces and provides ways to analyse mappings or groups acting on them.  Here we mostly follow the set up as in \cite{Ri}.

Let $(X,d)$ be a metric space and let $\mathbb{R}^X$ be the space of all real functions on $X$ equipped with the  topology of pointwise convergence. Fix  a point $b\in X$, called the {\em base point}, and write $\mathrm{Lip}^1_b(X)$ to denote the set of all functions $h\in\mathbb{R}^X$ such that $h(b)=0$ and $h$ is $1$-Lipschitz, i.e., $|h(x)-h(y)|\leq d(x,y)$ for all $x,y\in X$.

The set $\mathrm{Lip}^1_b(X)$ is a compact subset of $\mathbb{R}^X$. Indeed, it is easy to verify that the complement of $\mathrm{Lip}^1_b(X)$ is open. Moreover, as $|h(x)|= |h(x)-h(b)|\leq d(x,b)$ for all  $x\in X$ and $h\in \mathrm{Lip}^1_b(X)$, we have that $\mathrm{Lip}_b^1(X)\subseteq [-d(x,b),d(x,b)]^X$. The set $[-d(x,b),d(x,b)]^X$ is compact by Tychonoff's theorem, and hence $\mathrm{Lip}^1_b(X)$ is a compact subset of $\mathbb{R}^X$.

 For $y\in X$ define the real valued function,
\begin{equation}\label{internalpoint}
h_{y}(z) = d(z,y)-d(b,y)\qquad (z\in X).
\end{equation}
Note that $h_y(b)=0$ and $|h_y(z)-h_y(w)| = |d(z,y)-d(w,y)|\leq d(z,w)$ for all $z,w\in X$, and hence $h_y\in  \mathrm{Lip}_b^1(X)$.

\begin{definition}\label{metric_compactifcation} Denote the closure of $\{h_y\colon y\in X\}$ in $\mathbb{R}^X$ by  $\overline{X}$, which is compact. The set $X(\infty)= \overline{X}\setminus \{h_y\colon y\in X\}$ is called the {\em horofunction boundary of $(X,d)$}. The elements of $X(\infty)$ are called {\em horofunctions}, and the set $X\cup X(\infty)$ is called the {\em metric (or horofunction) compactification of $(X,d)$}. The elements of $\overline{X}$ are called {\em metric functionals}, and the metric functionals $h_y$  in (\ref{internalpoint}) are called {\em internal points}.
\end{definition}
The topology of pointwise convergence on  $ \mathrm{Lip}_b^1(X)$ coincides with the topology of uniform convergence on compact sets, see \cite[p.\,291]{Mun}. In general the topology of pointwise convergence on  $ \mathrm{Lip}_b^1(X)$ is not metrizable, and hence  horofunctions are limits of nets  rather than sequences. However, if the metric space is separable, then the topology is metrizable and each horofunction is the limit of a sequence. In fact, one can verify that  given a countable dense subset  $\{y_k\colon k\in\mathbb{N}\}$ of $(X,d)$, the function $\sigma$ on $\mathrm{Lip}_b^1(X)\times \mathrm{Lip}_b^1(X)$ given by,
\[
\sigma(f,g) = \sum_k 2^{-k}\min\{1,|f(y_k)-g(y_k)|\}\mbox{\quad  for $f,g\in  \mathrm{Lip}_b^1(X)$,}
\]
is a metric whose topology coincides with the pointwise convergence topology on $\mathrm{Lip}_b^1(X)$,  \cite[p.\,289, Ex.\,10*]{Mun}.

The metric compactification may not be a compactification in the usual topological sense, as the  embedding $\iota\colon y\in X\mapsto h_y \in   \mathrm{Lip}_b^1(X)$ may fail to have the necessary properties. However, the embedding $\iota\colon X\to  \iota(X)$ is always a continuous bijection. Indeed, if $x\in X$ and we consider a neighbourhood $U=\{h\in \mathrm{Lip}_b^1(X)\colon |h(y) - h_x(y)|<\epsilon ,\, y\in X \}$ of $h_x$ for $\epsilon>0$, then for $z\in X$ with $d(x,z)< \epsilon/2$ we have that
$
|h_z(y) - h_x(y) | \leq |d(y,z) -d(y,x)| + |d(b,z) - d(b,x)|\leq 2d(z,x) <\epsilon,
$
and hence $\iota(z) \in U$, which shows that $\iota$ is continuous. Moreover, if $x,z\in X$ and $h_x=h_z$, then $0 = (h_z(x) - h_x(x)) +(h_x(z) -h_z(z))  = 2d(x,z)$, which gives injectivity.

It can happen that  $\iota\colon X \to   \iota(X)$  does not have a continuous inverse.  If,  however,  $(X,d)$ is proper and geodesic, then the metric compactification will be a compactification in the usual topological sense. We provide some details of this fact below. Recall that a metric space $(X,d)$ is {\em proper} if all its closed balls are compact.  Note that a proper metric space is separable, as each closed ball is compact and hence separable.

A map $\gamma$ from a, possibly unbounded, interval $I\subseteq \mathbb{R}$ into a metric space $(X,d)$ is called a {\em geodesic path} if
\[
d(\gamma(s),\gamma(t)) = |s-t|\mbox{\quad for all }s,t\in I.
\]
The image, $\gamma(I)$, is called a {\em geodesic}. A metric space $(X,d)$ is said to be {\em geodesic} if for each $x,y\in X$ there exists a geodesic path $\gamma\colon [a,b]\to X$ connecting $x$ and $y$, i.e, $\gamma(a)=x$ and $\gamma(b) = y$.

Hermitian symmetric spaces with  Carath\'eodory distance are proper geodesic metric spaces (cf.\,\cite{gromov}). In the discussion below we will focus on the metric compactification of such metric spaces.

The horofunctions of a proper geodesic metric space $(X,d)$ are precisely the limits of converging sequences $(h_{x_k})$ such that $d(b,x_k)\to\infty$. A slightly stronger assertion was shown in \cite[Theorem 4.7]{Ri}, but for our purposes the following statement will suffice, see also \cite[Lemma 2.1]{LLN}.
\begin{lemma}\label{Rieffel}
If $(X,d)$ is a proper geodesic metric space, then $h\in  X(\infty)$ if and only if there exists a sequence $(x_k)$ in $X$ with $d(b,x_k)\to \infty$ such that $(h_{x_k})$ converges to $h\in \overline{X}$ as $k\to \infty$.
\end{lemma}

We  use this lemma to show that, in this case, the embedding $\iota\colon  X\to  \iota(X)$ has a continuous inverse.
\begin{lemma} If $(X,d)$ is a proper geodesic metric space, then $\iota\colon X\to  \iota(X)$  is a homeomorphism.
\end{lemma}
\begin{proof} From the previous observations it remains to show that $\iota \colon X\to  \iota(X)$ has a continuous inverse.
Let $h_{z_0} =\iota(z_0)$ where $z_0\in X$.  Note that as $(X,d)$ is proper it is also separable. So, to prove continuity of $\iota^{-1}$ at $h_{z_0}$, we can use sequences, as the topology of pointwise convergence on $\mathrm{Lip}_b^1(X)$ is metrizable. Let $(z_k)$ be a sequence in $X$ with $h_{z_k}\to h_{z_0}$. By Lemma \ref{Rieffel} we know that $(z_k)$ is bounded, and hence after taking a subsequence we may assume that $z_k\to z$. It follows that $h_z=h_{z_0}$, and hence the injectivity of $\iota$ implies that $z=z_0$, which completes the proof.
\end{proof}
Thus, the metric compactification is a compactification in the usual topological sense if $(X,d)$ is a proper geodesic space.

Special horofunctions come from so-called almost geodesics sequences. They were introduced by Rieffel \cite{Ri} and further developed by Walsh and co-workers in \cite{AGW,LW,Wa2,Wa1}.  A sequence $(x_k)$  in $(X,d)$ is called an {\em almost geodesic} if for each $\epsilon>0$ there exists an $N\geq 0$ such that
\[
d(x_n,x_m) +d(x_m,x_0) - d(x_n,x_0) <\epsilon\mbox{\quad for all }n\geq m\geq N.
\]
In particular, every unbounded almost geodesic sequence yields a horofunction for a proper geodesic metric space, see \cite{Ri}.
\begin{lemma} Let $(X,d)$ be a proper geodesic metric space. If $(x_k)$ is an unbounded almost geodesic in $(X,d)$, then
\[
h(z) = \lim_k d(z,x_k)-d(b,x_k)
\]
exists for all $z\in X$ and $h\in X(\infty)$.
\end{lemma}
Given a proper geodesic metric space $(X,d)$, a horofunction $h$ is called a {\em Busemann point} if there exists an almost geodesic  $(x_k)$ in $X$ such that $h(z) = \lim_k d(z,x_k)- d(b,x_k)$ for all $z\in X$. The collection of all Busemann points  is denoted by $\mathcal{B}_X$.

The set of Busemann points can be equipped with a metric known as the detour distance, which was introduced in \cite{AGW}, and is defined as follows. Suppose $(X,d)$ is a proper geodesic metric space and $h,h'\in X(\infty)$ are horofunctions. Let $W_h$ be the collection of neighbourhoods of $h$ in $\overline{X}$. Then the {\em detour cost} is given by
\[
H(h,h') = \sup_{W\in W_h}\left(\inf_{x\colon \iota(x)\in W} d(b,x)  +h'(x)\right),
\]
and the {\em detour distance} is defined by
\[
\delta(h,h') = H(h,h')+H(h',h).
\]

It is known \cite{LW,Wa1} that  if $(x_k)$ is an almost geodesic  converging to a horofunction $h$, then
\begin{equation}\label{detourcost}
H(h,h') = \lim_k d(b,x_k) +h'(x_k)
\end{equation}
for all horofunctions $h'$. Moreover, on the set of Busemann points $\mathcal{B}_X$ the detour distance  is a metric where points can be at infinite distance from each other, see \cite{LW,Wa1}.

The detour distance induces a partition of $\mathcal{B}_X$ into equivalence classes, called {\em parts}, where $h$ and $h'$ in $\mathcal{B}_X$ are equivalent if  $\delta(h,h') <\infty$. In particular, if all horofunctions are Busemann points, so $\mathcal{B}_M= X(\infty)$, then $X(\infty)$ is the disjoint union of parts, each of which is a metric space under the detour distance.

The horofunction  boundary $X(\infty)$ has a natural partition  induced by the equivalence relation:
\begin{equation*}\label{sim}
h\sim h'\mbox{\quad if \quad}\sup_{x\in X} |h(x) -h'(x)|<\infty.
\end{equation*}
It was shown in \cite[Proposition 4.5]{Wa1} that two Busemann points $h$ and $h'$ in $X(\infty)$ have finite detour distance if, and only if, $h\sim h'$. Thus,  if all horofunctions are Busemann points, then the partition of $X(\infty)$ into parts coincides with the partition into equivalence classes $X(\infty)/\sim$.

We also  like to note that each (surjective) isometry $\psi\colon X\to X$ extends as a homeomorphism to $X(\infty)$ by
\[
\psi(h)(x) = h(\psi^{-1}(x)) - h(\psi^{-1}(b))\qquad (x\in X, h\in X(\infty)).
\]
It is known that on the Busemann points in $X(\infty)$ the extension $\psi$ is an isometry under the detour distance, see e.g., \cite{LW}. Also if $h\sim g$ in $X(\infty)$, then $\psi(h)\sim \psi(g)$.

\section{Jordan algebraic structures}\label{jordan}

In this section, we recall some  necessary definitions and results concerning Jordan algebraic structures associated with Hermitian symmetric spaces.
Throughout, $M$ will denote an arbitrary Hermitian symmetric space of noncompact type. To determine the horofunctions of $M$ and analyse the geometry and global topology of its metric compactification with respect to
the Carath\'eodory distance (and base point $b\in M$), we make use of the fact that $M$ can be realised as the open unit ball in a $\JB^*$-triple. Indeed, there exists a  biholomorphism,
\begin{equation}\label{psi}
\psi\colon M~~ ^{\Large \underrightarrow{\psi_1}}~ ~\mathcal{D}~~ ^{\Large \underrightarrow{\psi_2}}~~ D \subset V,
\end{equation}
onto the open unit ball $D$ of a finite dimensional $\JB^*$-triple $V$, with $\psi(b)=0$, where $\psi_1$ is the Harish-Chandra embedding and $\psi_2$ is the Kaup Riemann mapping, which is unique up to a linear isometry \cite[Theorem 4.9]{kaup}.

Since the biholomorphism $\psi$ preserves the Carath\'eodory distance, it induces a homeomorphism between  the metric compactifications of $M$ and $D$.  Hence we can, and will, work in the setting
\[
D\subset V, \mbox{ where $D$ is the open unit ball of a finite dimensional $\JB^*$-triple $V$.}
\]
The results can be transferred to the corresponding ones for $M$ via $\psi$.\smallskip

 Let us now present some of the essential background on Jordan triple systems, where we focus on the finite dimensional case.  A {\em $\JB^*$-triple} is a complex Banach space $V$ equipped with
 a continuous triple product
\[
\{\cdot,\cdot,\cdot\} \colon V\times V \times V \to V,
\]
called a {\em Jordan triple product}, which is linear and symmetric in the outer variables, and conjugate linear in the
middle variable, and  satisfies the following axioms:
\begin{enumerate}
\item[(i)]$\{a,b,\{x,y,z\}\}= \{\{a,b,x\},y,z\}-\{x, \{b,a,y\},z\} + \{x,y,\{a,b,z\}\}$, \quad (Jordan triple identity)
\item[(ii)] $a\bo a$ is Hermitian, that is,
$\|\exp it(a\bo a)\|=1$ for all $t \in \mathbb{R}$;
\item[(iii)]
$a\bo a$ has nonnegative spectrum $\sigma(a\bo a)$;
\item[(iv)] $\|a\bo a\|=\|a\|^2, \qquad $
\end{enumerate}
for $a,b,x,y,z\in V$,
where $a\bo b: V\to V$ is a bounded linear map, called a {\it box operator},
defined by
\begin{equation}\label{eq box operator}
a\bo b(x) = \{a,b,x\} \qquad (x\in V)
\end{equation}
and condition (iv) can be replaced by
\begin{equation}\label{c*-identity}
\|\{a,a,a\}\|= \|a\|^3 \qquad  (a\in V).
\end{equation}

We note that the box operator in \eqref{eq box operator} satisfies $\|a\bo b\| \leq \|a\|\|b\|$, as the triple product is a contractive mapping, that is,
\begin{equation*} \|\{a,b,c\}\|\leq \|a\| \|b\| \|c\| \quad {\rm for~all}~ a,b,c\in V.
\end{equation*}

\begin{remark}\label{norm}\rm
By definition, a Hermitian operator $T\colon V \longrightarrow V$ has
real numerical range, which is the closed convex hull of its spectrum
$\sigma(T)$  and $\|T\| = \sup\{|\lambda|\colon \lambda \in \sigma(T)\}$ \cite[pp.\,46--54]{BD}.
In particular, given $a, b$ in a $\mathrm{JB}^*$-triple,
 (i) and (ii) above implies
$$\|a\bo a\| = \sup\{\lambda\colon \lambda \in \sigma(a\bo a)\}$$
 and $a\bo b + b\bo a$ is Hermitian. Further, if
 $\|a\bo b + b \bo a\|\leq 1$, then we have $\sigma(a\bo b + b \bo a) \subset[-1,1]$.
\end{remark}

A prime example of a JB*-triple is the space of $p\times q$ complex matrices $M_{p,q}(\mathbb{C})$ with  Jordan
triple product,
\[
\{A,B,C\} = \frac{1}{2}(AB^*C + CB^*A) \qquad (A,B,C \in M_{p,q}(\mathbb{C})),
\]
which has open unit ball $D=\{A\in M_{p,q}(\mathbb{C})\colon I-AA^*\mbox{ positive definite}\}$.  In particular if $q=1$, we get the complex Euclidean space $\mathbb{C}^p$ with Jordan triple product
\[
\{x,y,z\} = \frac{1}{2}(\langle x,y\rangle z + \langle z,y\rangle x) \qquad (x,y,z \in \mathbb{C}^p)
\]
and $D =\{z\in \mathbb{C}^p\colon \langle z,z\rangle <1\}$ is the Euclidean ball.

Given $\JB^*$-triples $V_1,\ldots,V_d$,  the direct sum $V_1\oplus\cdots\oplus V_d$, with the $\ell_\infty$-norm,
	$$\| (a_1,\ldots,a_d)\|_\infty = \max\{\|a_i\|\colon i=1,\ldots,d\}, \quad (a_i\in V_i)$$
is a $\JB^*$-triple with the coordinatewise triple product.


In fact, a finite dimensional $\JB^*$-triple $V$  decomposes into a finite direct sum $V_1\oplus \cdots\oplus V_d$ of so-called Cartan factors $V_j$ ($j=1,\dots, d$) with $\ell_\infty$-norm. There are six different types of (finite dimensional) Cartan factors:

(1) $M_{p,q}(\mathbb{C})$ \qquad (2) $S_q(\mathbb{C})$ \qquad (3) $H_q(\mathbb{C})$\qquad (4) $Sp_n(\mathbb{C})$\qquad (5) $M_{1,2}(\mathcal{O})$   \qquad (6) $H_3(\mathcal{O})$,\\\noindent
where $S_q(\mathbb{C})$ and $H_q(\mathbb{C})$ are norm closed subspaces
 of $M_{p,q}(\mathbb{C})$ consisting of $p\times q$ skew-symmetric and symmetric matrices, respectively;
 and $Sp_n(\mathbb{C})$ is a spin factor of dimension $n>2$. The Cartan factors of types 5 and 6 are \emph{exceptional} Cartan factors
  (cf.\,\cite[Theorem 2.5.9]{book1}).\smallskip



There are various operators that play an important role in the theory of $\JB^*$-triples.
Besides the box operators, we will use the {\it Bergman operator}
$B(b,c)\colon V\to V$ and the {\it M\"obius transformation} $g_a\colon D\to D$, where $a\in D$
and $b,c \in V$, which are defined as follows:
\begin{equation}\label{b1}
B(b,c)(x) = x - 2(b\bo c)(x)+\{b,\{c,x,c\},b\} \qquad (x\in V),
\end{equation}
\begin{equation}\label{g2}
g_a(x) = a + B(a,a)^{1/2}(I+x\bo a)^{-1}(x) \qquad (x\in D).
\end{equation}
Here $I$ denotes the identity operator on $V$, and the inverse
$(I+x\bo a)^{-1}\colon  V \to V$ exists, as $\|x\bo a\|\leq \|x\|\|a\| <1$.

We note that $B(a,b)$ is invertible for $\|a\|\|b\|<1$. The proof of the following two identities
 can be found in \cite[Proposition 3.2.13. Lemma 3.2.17]{book1}.
\begin{equation}\label{bid2}
\|B(z,z)^{-1/2}\| = \frac{1}{1-\|z\|^2} \qquad (\|z\|<1),
\end{equation}
\begin{equation}\label{bid}
1- \|g_{-y}(z)\|^2 = \frac{1}{\|B(z,z)^{-1/2}B(z, y)B(y, y)^{-1/2}\|}
\qquad (y, z\in D).
\end{equation}
For the Euclidean ball $D\subset \mathbb{C}^d$ with inner product $\langle \cdot,\cdot\rangle$,
 we have from \cite[Example 3.2.29]{book2} the formula
\begin{equation}\label{hh}
1-\|g_{-y}(z)\|^2 = \frac{(1-\|y\|^2)(1-\|z\|^2)}{|1-\langle y,z\rangle|^2} \qquad (y,z\in D).
\end{equation}

Given $a\in V$, the {\em quadratic operator} $Q_a \colon V \to V$ is defined  by
$$Q_a(x) = \{a,x,a\} \qquad (x\in V).$$

An element $e$ in a $\JB^*$-triple $V$ is called a {\it tripotent} if $\{e,e,e\}=e$. Although $0$ is a tripotent
in a $\JB^*$-triple, we are only interested in the nonzero ones, of which the norm is always $1$.
Tripotents in C*-algebras are exactly the partial isometries.

Any tripotent $e$ in $V$ induces an eigenspace decomposition of $V$, called the {\it Peirce decomposition associated with $e$}. The eigenvalues of the box operator $e\bo e\colon V \to V$ are in
the set $\{0, 1/2,1\}$. Let
$$V_k(e) = \{x\in V\colon (e\bo e)(x)=\frac{k}{2} x\} \qquad (k=0,1,2)$$
be the corresponding eigenspaces, called the {\it Peirce $k$-space} of $e$. We have the algebraic direct sum
$$V= V_0(e) \oplus V_1(e) \oplus V_2(e).$$
where the Peirce $k$-spaces satisfy
\begin{equation}\label{eq Peirce arit 1}
	\{V_{i}(e),V_{j} (e),V_{k} (e)\}\subseteq V_{i-j+k} (e)
\end{equation} if $i-j+k$ belongs to the set  $\{ 0,1,2\},$
and $\{V_{i}(e),V_{j} (e),V_{k} (e)\}=\{0\}$ otherwise. Further, we have
\begin{equation}\label{eq Peirce arit 2} \{ V_{2} (e), V_{0}(e), V\} = \{ V_{0} (e), V_{2}(e), V\} =0.
\end{equation}

The Peirce $k$-space $V_k(e)$ is the range of the {\it Peirce $k$-projection} $P_k(e) \colon V \to V$,
which are contractive and given by
$$P_2(e) = Q_e^2, \quad P_1(e) = 2(e\bo e -Q_e^2), \quad P_0(e) = B(e,e).$$

A tripotent $e$ in a JB*-triple $V$  is called {\it minimal} if $Q_e(V)= \mathbb{C}e$,
 or equivalently, $V_2(e)=\mathbb{C}e$.  It is called {\it maximal} if $V_0(e)=\{0\}$.
In fact, the maximal tripotents in $V$ coincide with  the extreme points of the closed unit ball $V$
(cf.\,\cite[Theorem 3.2.3]{book1}).

We note that, with the inherited norm from $V$,
 the Peirce 2-space $V_2(e)$ is a  $\JB^*$-algebra with identity $e$,  Jordan product and involution
\begin{equation}\label{jb*}
x\circ y =\{x,e,y\}, \quad  x^*= \{e,x,e\}=Q_ex, \quad (x,y\in V_2(e))
\end{equation}
respectively \cite[ Example 2.4.18]{book2}. In particular, we have
\begin{equation}\label{x*norm}
\|x\|=\|x^*\|=\|Q_ex\| \qquad (x\in V_2(e)).
\end{equation}
We refer  to \cite[Definition 2.4.16]{book2} for the definition of a
JB*-algebra, which are examples of JB$^*$-triples \cite[Lemma 3.1.6]{book1}.

Let
\begin{equation}\label{ae}
A(e)= \{ x\in V_2(e)\colon x^*=x\} = \{x\in V_2(e)\colon \{e,x,e\} =x\}
\end{equation}
 be the {\it self-adjoint part} of $V_2(e)$. Then it is a closed real subalgebra of $(V_2(e), \circ)$
satisfying  $$\|a\|^2 =\|a^2\|  \leq \|a^2 + b^2\|, \qquad (a,b\in A(e))$$ where $a^2= a \circ a$, in other words,
it is a so-called {\it JB-algebra} \cite[3.1.4]{hs}.

There is a natural partial ordering $\leq$ on $A(e)$
defined by the closed cone
$$A(e)_+= \{x^2: x\in A(e)\}$$
where $x\leq y$ if and only if $y-x\in A(e)_+$.
We will make use of the fact that $\{a, A(e)_+, a\} \subset A(e)_+$ for all $a\in A(e)$, and
\begin{equation}\label{a<e}
\|a\|\leq 1 \mbox{ if and only if } -e \leq a \leq e
\end{equation}
(cf.\,\cite[Proposition 3.3.6; 3.1.5]{hs}).
An element $a\in V_2(e)$ is called {\it invertible} if there is a (unique) element $a^{-1}$,
called the {\it inverse} of $a$,
 such that
$a\circ a^{-1}=e$ and $a^2 \circ a^{-1}=a$. If $a\in A(e)$, then $a^{-1}\in A(e)$.

Given $a,b\in V$, we say that $a$ is \emph{orthogonal} to $b$ if  $a \bo b =0$. It is known that $a$ is orthogonal to $b$ if and only if $\{a,a,b\} =0$. Moreover, $a$ orthogonal to $b$ implies $b$ orthogonal to $a$, in which case we have
$$\|a+ b\|= \max\{\|a\|, \|b\|\}$$
from \cite[Corollary 3.1.21]{book1}.

A linear subspace $W \subset V$ of a  JB$^*$-triple $V$ is called a JB$^*$-subtriple if $x,y,z \in W$ implies $\{x,y,z\}\in W$, in the inherited Jordan triple product. 

The rank $r$ of a  finite dimensional JB$^*$-triple $V$  is defined by
$$r=\sup \{ \hbox{dim}V(a) \colon a\in V \},$$
where $V(a)$ denotes the smallest closed subtriple of $V$ containing $a\in V$.
It can be shown that $r$
 is the maximal number of mutually orthogonal tripotents in $V$ \cite[Example 3.3.3]{book2}.

 Let $\{e_{1},\dots,e_{n}\}$ a family  of mutually orthogonal tripotents in a JB$^*$-triple $V$. For $i,j \in \{0,1, \ldots,n\}$, the {\it joint} Peirce space $V_{ij}$ is defined by
\begin{eqnarray}\label{ij}
V_{ij}=  V_{ij}(e_{1},\dots,e_{n})  =\{z\in V\colon 2\{e_{k},e_{k},z\}=(\delta_{ik}+\delta_{jk})z\mbox{ for }k=1,\ldots,n\},
\end{eqnarray}
where $\delta_{ij}$ is the Kronecker delta and $V_{ij}=V_{ji}$.

The decomposition
\begin{eqnarray*}
V & = & \bigoplus_{0\leq i\leq j\leq n}V_{ij}
\end{eqnarray*}
is called a {\it joint} Peirce decomposition.

The Peirce multiplication rules
\begin{eqnarray*}
\{V_{ij},V_{jk},V_{k\ell}\} & \subset & V_{i\ell}\q {\rm and} \q V_{ij} \bo V_{pq} =\{0\} \q {\rm for} \q
i,j \notin \{p,q\}
\end{eqnarray*}
hold.
  The contractive projection $P_{ij}(e_{1},\ldots,e_{n})$ from $V$ onto $V_{ij}(e_{1},\ldots,e_{n})$
is called a {\it joint Peirce projection} which satisfies
\begin{equation}
 P_{ij}(e_{1},\dots,e_{n})(e_{k})= \left\{\begin{array}{ll} 0& (i \neq j)\\
                                                   \delta_{ik} e_k & (i=j).\label{pijek}
                                         \end{array}\right.
\end{equation}
We shall simplify the notation $P_{ij}(e,\ldots,e_n)$ to $P_{ij}$ if the tripotents $e_1, \ldots, e_n$
are understood.

Let $M=\{0,1,\ldots,n\}$ and $N\subset\{1,\ldots,n\}$. The Peirce $k$-spaces of the tripotent $e_{N}=\sum_{i\in N}e_{i}$
 are given by
\begin{eqnarray}
V_{2}(e_{N}) & = & \bigoplus_{i,j\in N}V_{ij},\label{Peirce-2 of sum}\\
V_{1}(e_{N}) & = & \bigoplus_{\substack{i\in N\\
j\in M\backslash N
}
}V_{ij}, \nonumber\\
V_{0}(e_{N}) & = & \bigoplus_{i,j\in M\backslash N}V_{ij}.\nonumber \label{eq:Peirce-0 of sum}
\end{eqnarray}

The Peirce projections provide a very useful formulation of the Bergman operators.
Let $e_{1},\ldots,e_{n}$
be  mutually orthogonal tripotents in a JB*-triple $V$ and
let $x=\sum_{i=1}^{n}\lambda_{i}e_{i}$ with
$\lambda_{i}\in\mathbb{C}$.
Then the Bergmann
operator $B(x,x)$ satisfies
\begin{eqnarray*}
B(x,x) & = & \sum_{0\leq i\leq j\leq n}(1-|\lambda_{i}|^{2})(1-|\lambda_{j}|^{2})P_{ij},\label{eq:Loos Bergmann}
\end{eqnarray*}
where we set $\lambda_{0}=0$ and $P_{ij}=P_{ij}(e_1, \ldots,e_n)$.  This gives the following formulae
for the square roots
\begin{eqnarray}
B(x,x)^{1/2} & = & \sum_{0\leq i\leq j\leq n}(1-|\lambda_{i}|^{2})^{1/2}(1-|\lambda_{j}|^{2})^{1/2}P_{ij}\quad (\|x\|<1), \label{eq:Bergmann Sq Rt}\\
B(x,x)^{-1/2} & = & \sum_{0\leq i\leq j\leq n}(1-|\lambda_{i}|^{2})^{-1/2}(1-|\lambda_{j}|^{2})^{-1/2}P_{ij} \quad (\|x\|<1).\label{eq:Bergmann Neg Sq Rt}
\end{eqnarray}

The following lemma will be useful later for computing the horofunctions in a noncompact Hermitian symmetric space.
 \begin{lemma}\label{hhh} Let $D$ be  the open unit ball of a JB*-triple $V$.
Given a sequence $(y_k)$ in $D$ such that $y_k\to \xi\in\partial D$, we have
\begin{equation*}\label{z}
\lim_{k}\|g_{-y_k}(z)\|=1 \qquad (z\in D).
\end{equation*}
\end{lemma}
\begin{proof}
By (\ref{bid2}) and (\ref{bid}), we have
\begin{eqnarray*}
0<1- \|g_{-y_k}(z)\|^2 &=& \frac{1}{\|B(z,z)^{-1/2}B(z, y_n)B(y_k, y_k)^{-1/2}\|}\\
&\leq& \frac{\|B(z, y_k)^{-1}B(z, z)^{1/2}\|}{\|B(y_k, y_k)^{-1/2}\|}\\
&=& \|B(z, y_k)^{-1}B(z, z)^{1/2}\|(1-\|y_k\|^2) \longrightarrow 0 \qquad (z\in D)
\end{eqnarray*}
as $k \rightarrow \infty$, since $\lim_k \|B(z, y_k)^{-1}B(z, z)^{1/2}\|=
\|B(z,\xi)^{-1}B(z, z)^{1/2}\|$.
\end{proof}

For more details of JB*-triples,  we refer to \cite{book2} and the references therein.

 \section{Horofunctions of Hermitian symmetric spaces}\label{HermHor}

 Given the open unit ball $D$ in a finite dimensional $\JB^*$-triple $V$, we now determine the horofunctions of $D$ under the Carath\'eodory distance $\rho$, with the origin $0\in D$ as a base point .  Recall that the horofunctions of the corresponding Hermitian symmetric space $M$
  with base point $b\in M$ can be obtained via  $\psi$ in (\ref{psi}), with $\psi(b)=0$, where $\psi$ preserves the Carath\'eodory distance.
 In fact,  we have
$$M(\infty) = \{h\circ \psi\colon  h\in D(\infty)\}.$$

The Carath\'eodory distance $\rho$ on  $D$ is given by
$$\rho(x,y) = \sup\{ \omega(f(x),f(y))\colon  f \in H(D, \mathbb{D})\} \qquad (x,y\in D),$$
where $H(D,\mathbb{D})$ is the set of all holomorphic functions $f\colon  D\longrightarrow \mathbb{D}$
and $\omega$  is the Poincar\'e distance of the unit disc $\mathbb{D}=\{z\in \mathbb{C}\colon |z|<1\}$.
We will make use of the formula (cf.\,\cite[Theorem 3.5.9]{book2}):
$$\rho(x,y) = \tanh^{-1}\|g_{-x}(y)\|.$$

 For each $y,z\in D$ we have
 \begin{equation}\label{g}
 h_y(z)  =  \rho(z,y)-\rho(0,y) = \frac{1}{2}\log \frac{1+ \|g_{-y}(z)\|}{1-\|g_{-y}(z)\|} -
 \frac{1}{2}\log \frac{1+ \|y\|}{1-\|y\|}
= \frac{1}{2}\log \left(\frac{1- \|y\|^2}{1-\|g_{-y}(z)\|^2}\left(\frac{1+\|g_{-y}(z)\|}{1+ \|y\|}\right)^2 \right),
\end{equation}
which  can also be written as
\begin{equation}\label{gg}
h_y(z) =  \frac{1}{2}\log \left(\frac{1- \|y\|^2}{1-\|g_{-z}(y)\|^2}
\left(\frac{1+\|g_{-z}(y)\|}{1+ \|y\|}\right)^2 \right).
\end{equation}

\begin{lemma}\label{d2}
Let $D$ be the open unit in a finite dimensional $\JB^*$-triple $V$, and $\rho$ the Carath\'eodory distance on $D$.
Then  $h \in  D(\infty)$ if and only if  there exists a sequence  $(y_k)$ in $D$ with $y_k\to \xi \in\partial D$ such that
\begin{equation}\label{hd}
h(z) = \lim_k \frac{1}{2}\log \left(\frac{1- \|y_k\|^2}{1-\|g_{-y_k}(z)\|^2}\right)
= \lim_k \frac{1}{2}\log \left(\frac{1- \|y_k\|^2}{1-\|g_{-z}(y_k)\|^2}\right)
\end{equation}
for all $z\in D$.
\end{lemma}
\begin{proof} Suppose that $h\in D(\infty)$. Then by Lemma \ref{Rieffel} we know that there exists a sequence $(y_k)$ in $D$ with $\rho(0,y_k)\to \infty$ and $\lim_k h_{y_k}(z)$ exists for all $z\in D$.  By taking a subsequence we may assume that $y_k\to\xi\in\partial D$. The implication now follows from (\ref{g}) and Lemma \ref{hhh}. On the other hand, if  there exists a sequence  $(y_k)$ in $D$ with $y_k\to \xi \in\partial D$ such that (\ref{hd}) holds for all $z\in D$, then $\rho(0, y_k) = \tanh^{-1}\|y_k\|\to \infty$, as $\|y_k\|\to 1$. So, we deduce from (\ref{g}) and Lemma \ref{Rieffel} that $h\in D(\infty)$.
\end{proof}
\begin{remark} The notion of  a horofunction $h$ on  $D$ is essentially the same as the function $F$
 introduced in \cite[Lemma 4.1]{cr}. Indeed, the formula (\ref{hd}) in Lemma \ref{d2} for $h$ is related to $F$ by
\begin{equation*}\label{hF}
h (x) = \frac{1}{2} \log F(x) \qquad (x\in D).
\end{equation*}
\end{remark}

In case $D$ is the open Euclidean ball in  $\mathbb{C}^d$ and $y_k\to \xi$ with $\|\xi\|_2=1$, we find that
$$h(z) = \lim_k \frac{1}{2}\log \left(\frac{1- \|y_k\|^2}{1-\|g_{-y_k}(z)\|^2}\right)$$
is a horofunction, and from (\ref{hh}) we have
$$h(z) = \lim_k \frac{1}{2}\log \left(\frac{|1-\langle z, y_k\rangle|^2}{1- \|z\|^2}\right)
=\frac{1}{2}\log \left(\frac{|1-\langle z, \xi\rangle|^2}{1- \|z\|^2}\right).$$
In particular, for the disc $\mathbb{D}\subset \mathbb{C}$  we find that $h\colon \mathbb{D}\longrightarrow \mathbb{R}$ is given by the well known expression,
 $$ h(z) = \frac{1}{2}\log  \frac{|1- z\overline\xi|^2}{1-|z|^2}
=\frac{1}{2}\log  \frac{|\xi- z|^2}{1-|z|^2} \qquad (z\in \mathbb{D}).$$

We now compute the limit in (\ref{hd}) for general $D$.
By \cite[Lemma 3.2.28]{book2}, one can express the limit $h= \lim_k h_{y_k}$ in (\ref{hd})
 in terms of the Bergman operators as
\begin{equation*}\label{hd+}
h(z) = \lim_{k\rightarrow \infty} \frac{1}{2}\log (1-\|y_k\|^2)\|B(z,z)^{-1/2}B(z,y_k)B(y_k,y_k)^{-1/2}\| \qquad (z\in D).
\end{equation*}

Let $r$ be the rank of $V$. Each $y_k\in D$ has a spectral
decomposition
\begin{equation}\label{sd}
y_k = \alpha_{1k}e_{1k} + \cdots + \beta_{dk}e_{rk}
\end{equation}
where $e_{1k}, \ldots, e_{rk}$ are mutually orthogonal minimal tripotents in $V$ and
$$1>\|y_k\|= \alpha_{1k} \geq \alpha_{2k} \geq \cdots \geq \alpha_{rk} \geq 0.$$
Choosing a subsequence, we may assume for each $i$ that the sequence $(\alpha_{ik})$ converges to some $\alpha_i \in [0,1]$ and  the minimal tripotents $e_{ik}$ converge to minimal tripotent $e_i$, as the set of minimal tripotents is a closed subset of $\partial D$. Note that  $\alpha_1 = \lim_{k} \|y_k\| =1$, as $\|y_k\|\to\|\xi \| =1$.

By \cite[Lemma 5.8]{cr}, we have
$$\xi = \lim_k y_k = \alpha_1e_1 +\cdots +\alpha_r e_r$$
 and  there exists $r_0 \in \{1, \ldots, r\}$ such that
\begin{enumerate}
\item[(i)] $\alpha_s >0$ for each $1\leq s\leq r_0$,
\item[(ii)] $\alpha_s =0$ for $s>r_0$,
\item[(iii)]$\{e_1,\ldots,e_{r_0}\}$ is family of mutually orthogonal minimal tripotents.
\end{enumerate}

\begin{remark}\label{pij}\rm
The minimal tripotents $e_{1k}, \ldots, e_{rk}$ in the spectral decomposition of $y_k$ induce a joint Peirce
decomposition of $V$, with joint Peirce projections $P^k_{ij}=P_{ij}(e_{1k}, \ldots, e_{rk})$ and $0\leq i\leq j\leq r$.
By (\ref{eq:Bergmann Neg Sq Rt}), the Bergman operator $B(y_k,y_k)^{-1/2}$ is of the form
$$B(y_k,y_k)^{-1/2} = \sum_{0\leq i\leq j\leq r} (1-\alpha_{ik}^2)^{-1/2}(1-\alpha_{jk}^2)^{-1/2}P_{ij}^k \quad (\alpha_{0k}=0).$$
As $(e_{sk})_k$ converges to a minimal tripotent $e_s$ for $1\leq s\leq r_0$, with $r_0$ as above, and the $e_s$'s are pairwise orthogonal, we have the following norm convergence,
$$\lim_{k\rightarrow \infty} P_{ij}(e_{1k},\ldots,e_{r_0k}) = P_{ij}(e_1,\ldots,e_{r_0})$$
of  Peirce projections (cf.\,\cite[Remark 5.9]{cr}).

Furthermore, if $e_1,\ldots,e_m$ are mutually orthogonal tripotents and $1\leq q<m$, then
$P_{ij}(e_1,\ldots,e_q)=P_{ij}(e_{1}, \ldots, e_{m})$ for all $1\leq i\leq j\leq q$, see \cite[Lemma 2.1(i)]{cr}.
\end{remark}
We will use the observations in the previous remark to prove the following theorem.
\begin{theorem}\label{hbsd} Let $D$ be the open unit
ball of a finite dimensional $\JB^*$-triple $V$, with rank $r$. Then the horofunction functions in $D(\infty)$ are exactly the functions
 of the form
$$\label{horD}
h(z) =\frac{1}{2} \log \left\|\sum_{1\leq i\leq j\leq p}\lambda_i\lambda_jB(z,z)^{-1/2}B(z,e)P_{ij}\right\|
\qquad (z\in D),
$$
where $p\in \{1, \ldots, r\}$,  $\lambda_i\in (0,1] ~(i=1,\ldots,p)$ with $\max_i \lambda_i =1$, $e = e_1 + e_2 + \cdots +  e_{p}\in \partial D$,
and $P_{ij}\colon V \rightarrow V$ are the Peirce projections induced by the mutually orthogonal minimal tripotents $e_1, \ldots, e_{p}$.
\end{theorem}
\begin{proof}
Suppose that $h$ is a horofunction. Then by Lemma \ref{d2} there exists a sequence $(y_k)$ in $D$ converging to $\xi\in\partial D$ such that $h(z)= \lim_{k\rightarrow \infty} h_{y_k}(z)$ for all $z\in D$.

Let $$y_k = \alpha_{1k}e_{1k} + \cdots + \alpha_{rk}e_{rk}$$ be the spectral decomposition. From Remark \ref{pij}  we know
there is an $r_0\in \{1,\ldots,r\}$ such that
$$\xi = \lim_k y_k = \alpha_1e_1 +\cdots +\alpha_r e_r$$
with $\alpha_1=1$, $\alpha_s>0$ for $1\leq s\leq r_0$ and $\alpha_s=0$ for $s>r_0$.
Moreover,
$$B(y_k,y_k)^{-1/2} = \sum_{0\leq i\leq j\leq r} (1-\alpha_{ik}^2)^{-1/2}(1-\alpha_{jk}^2)^{-1/2}P_{ij}^k \quad (\alpha_{_{0k}}=0).$$

Since $0<{1-\alpha_{1k}^2}\leq 1-\alpha_{ik}^2$ for $ i \in \{1, \ldots,r\}$, we may assume, by choosing subsequence if necessary,
for each $i$ that
$$\frac{1-\alpha_{1k}^2}{1-\alpha_{ik}^2} ~ {\rm converges~ to~ some~} \lambda_i \in [0,1].$$
Note that  $\lambda_1 =1$ and  $\lambda_i =0$ for $i >r_0$.  Combining this with Remark \ref{pij} we get that
\begin{eqnarray*}
\lim_{k\rightarrow \infty} (1-\|y_k\|^2)B(y_k, y_k)^{-1/2} &= &
\lim_{k\rightarrow \infty} \sum_{0\leq i\leq j\leq r}
\sqrt{\frac{1-\alpha_{1k}^2}{1-\alpha_{ik}^2}}
\sqrt{\frac{1-\alpha_{1k}^2}{1-\alpha_{jk}^2}} \,P_{ij}^k\\
&=& \lim_{k\rightarrow \infty} \sum_{0\leq i\leq j\leq r_0}
\sqrt{\frac{1-\alpha_{1k}^2}{1-\alpha_{ik}^2}}
\sqrt{\frac{1-\alpha_{1k}^2}{1-\alpha_{jk}^2}}\, P_{ij}(e_{1k},\ldots,e_{r_0k}) \\
&=&  \sum_{0\leq i\leq j\leq r_0}
\lambda_i \lambda_j P_{ij}
\end{eqnarray*}
where $P_{ij}$ are the Peirce  projections induced by the orthogonal minimal tripotents
$e_1, \ldots, e_{r_0}$. So,
\begin{eqnarray}\label{1st}
 h(z)& = &\lim_{k\rightarrow \infty} \frac{1}{2}\log (1-\|y_k\|^2)\|B(z,z)^{-1/2}B(z,y_k)B(y_k,y_k)^{-1/2}\| \nonumber\\
&= &\frac{1}{2} \log \left\|\sum_{1\leq i\leq j\leq r_0}\lambda_i\lambda_jB(z,z)^{-1/2}B(z,\xi)P_{ij}(e_1, \ldots, e_{r_0})\right\|
\qquad (z\in D).
\end{eqnarray}

Let $p\in\{1,\ldots,r_0\}$ be such  that  $\alpha_i=1$ for $i\leq p$, and $\alpha_i<1$ otherwise.  Since  $\lambda_i=0$ when $\alpha_i <1$, the horofunction $h$ in (\ref{1st}) reduces to
\begin{equation}\label{2nd}
h(z) =\frac{1}{2} \log \left\|\sum_{1\leq i\leq j\leq p}\lambda_i\lambda_jB(z,z)^{-1/2}B(z,\xi)P_{ij}(e_1,\ldots, e_{p})\right\|,
\end{equation}
as $P_{ij}(e_1,\ldots, e_{p}) = P_{ij}(e_1,\ldots,e_{r_0})$ for $1\leq i\leq j\leq p$ by \cite[Lemma 2.1]{cr}.

Let $e = e_1+\cdots +e_{p}$. For  $k\not\in \{i,j\}$ and $w\in V_{ij}(e_1, \ldots, e_{r_0})$ we have that $(e_k\bo w)( V) = \{0\}$  by the Peirce multiplication rules, as $e_k \in V_{kk}(e_1, \ldots, e_{r_0})$.
Therefore, for $1\leq i\leq j\leq p$,
\begin{equation*}
\xi \bo P_{ij}(e_1,\ldots,e_{p})(\cdot) =  \xi\bo P_{ij}(e_1, \ldots, e_{r_0})(\cdot)=
e \bo P_{ij}(e_1, \ldots, e_{r_0})(\cdot)= e \bo P_{ij}(e_1,\ldots,e_{p})(\cdot),
\end{equation*}
and likewise
\begin{eqnarray*}
Q_{\xi}P_{ij}(e_1,\dots,e_{p})(\cdot)  &= & \{e_1 + \alpha_2 e_2 + \cdots + \alpha_{r_0}e_{r_0},\, P_{ij}(e_1,\ldots,e_{p})(\cdot),\,
e_1 + \alpha_2 e_2 + \cdots + \alpha_{r_0}e_{r_0}\}\\
&=& \{ e, P_{ij}(e_1,\ldots,e_{p})(\cdot),e\}.
\end{eqnarray*}
Thus, $Q_{\xi}P_{ij}(e_1,\dots,e_{p}) = Q_{e}P_{ij}(e_1,\dots,e_{p}) $.
It now follows for $1\leq i\leq j\leq p$ that
\begin{equation*}
B(z,\xi)P_{ij}(e_1,\ldots,e_{p}) = P_{ij}(e_1,\ldots,e_{p}) - 2(z \bo e)P_{ij}(e_1,\ldots,e_{p}) + Q_zQ_{e}P_{ij}(e_1,\ldots,e_{p})\\
= B(z,e)P_{ij}(e_1,\ldots,e_{p}).
\end{equation*}
Thus, the horofunction $h$ in (\ref{2nd}) can be expressed as
$$h(z) = \frac{1}{2} \log \left\|\sum_{1\leq i\leq j\leq p}\lambda_i\lambda_jB(z,z)^{-1/2}B(z,e)P_{ij}(e_1,\ldots,e_{p})\right\|
\qquad (z\in D).$$

To prove that each function of the form (\ref{horD}) is a horofunction we let
$$h(z)= \frac{1}{2} \log \left\|\sum_{1\leq i\leq j\leq p}\lambda_i\lambda_jB(z,z)^{-1/2}B(z,e)P_{ij}\right\| \qquad (z\in D)$$
where $p\in \{1, \ldots, r\}$,  $\lambda_i\in (0,1] ~(i=1,\ldots,p)$ with $\max_i \lambda_i =1$, and $e= e_1 + e_2 + \cdots +  e_{p}\in \partial D$,
with Peirce projections $P_{ij}\colon V \rightarrow V$  induced by the mutually orthogonal minimal tripotents
$e_1, \ldots, e_{p}$.

For all $k \in \mathbb{N}$ sufficiently large (depends on $\min_i \lambda_i$) we can define
$$\alpha_{ik} = \sqrt{1- \frac{2k-1}{k^2\lambda_i^2}}$$  for $1\leq i\leq p$. For those $k$ set
$y_k = (1- 1/k)e_1 + \alpha_{2k}e_2 + \cdots + \alpha_{pk}e_{p}$ and note that $y_k\in D$.

Then the sequence $(y_k)$ norm converges to $e$ and
\begin{equation*}
\lim_{k\rightarrow \infty} (1-\|y_k\|^2)B(y_k, y_k)^{-1/2}  =
\lim_{k\rightarrow \infty} \sum_{0\leq i\leq j\leq p}
\sqrt{\frac{1-(1- 1/k)^2}{1-\alpha_{ik}^2}}
\sqrt{\frac{1-(1-1/k)^2}{1-\alpha_{jk}^2}} \,P_{ij}
=  \sum_{0\leq i\leq j\leq p} \lambda_i \lambda_j P_{ij}.
\end{equation*}
Hence
\begin{eqnarray*}
\lim_{k\rightarrow \infty} h_{y_k}(z) &=& \lim_{k\rightarrow \infty}   \frac{1}{2}\log \left(\frac{1- \|y_k\|^2}{1-\|g_{-y_k}(z)\|^2}\right)\\
&=& \lim_{k\rightarrow \infty}   \frac{1}{2}\log (1-\|y_k\|^2)\|B(z,z)^{-1/2}B(z,y_k)B(y_k,y_k)^{-1/2}\| \\
&=& \frac{1}{2} \log \left\|\sum_{1\leq i\leq j\leq p}\lambda_i\lambda_jB(z,z)^{-1/2}B(z,e)P_{ij}\right\|,
\end{eqnarray*}
which completes the proof.
\end{proof}

We see from the proof of Theorem \ref{hbsd} that it can happen that two sequence $(y_k)$ and $(z_k)$ in $D$ converging to distinct points in $\partial D$ can give the same horofunction. Indeed, if we let $y_k$ be as in the proof  Theorem \ref{hbsd}  and  set $z_k = y_k +(1-1/\sqrt{k})(e_{p+1} +\cdots+e_r)$, then both $h_{y_k}$ and $h_{z_k}$ converge to $h$ given by (\ref{horD}).

We also note that the horofunction $h$ given by (\ref{horD}) can be obtained by taking the limit of an appropriate sequence in the flat
$\mathbb{R}e_1\oplus \cdots\oplus \mathbb{R}e_p$.  This is consistent with the observation in \cite[Lemma 4.4]{HSWW}.  Later in Lemma \ref{hisBus} we shall show that one can obtain the horofunctions in $D(\infty)$ by taking  limits along  geodesics in the flats.

\section{Horofunctions of finite dimensional $\JB^*$-triples}\label{TripleHor}

We  now determine the horofunctions of  finite dimensional $\JB^*$-triples $(V,\|\cdot\|)$ as normed spaces, with base point $0$.
Throughout  we let $r$ be the rank of $V$. As in (\ref{sd}),  each element $a\in V$ has a spectral decomposisiton,
$$a = \lambda_1 e_1 +  \lambda_2 e_2+ \cdots + \lambda_r e_r, \quad
(\|a\|= \lambda_1 \geq \lambda_2 \geq \cdots \geq \lambda_r \geq 0),$$
where $e_1, \ldots, e_r$ are mutually orthogonal minimal tripotents in $V$.

Given a sequence $(a_k)$ in $V$ with $h_{a_k}\to h\in V(\infty)$, we have  $r_k=\|a_k\|\to \infty$ (by Lemma \ref{Rieffel}) and
\[
h_{a_k}(x) =\|x-a_k\| - \|a_k\|= \frac{\|x-a_k\|^2-\|a_k\|^2}{\|x-a_k\|+\|a_k\|} =  \frac{(2r_k)^{-1}(\|(x-a_k)\bo (x-a_k)\|-r_k^2)}{2^{-1}(\|r_k^{-1}(x-a_k)\|+1)}.
\]
As the denominator goes to $1$ when $k\to\infty$, we need to analyse
\[
\lim_{k\to\infty} (2r_k)^{-1}(\|(x-a_k)\bo (x-a_k)\|-r_k^2).
\]
First we note that,  for each  $y\in V$,  the spectrum $\sigma(y\bo y)$ is the set of
eigenvalues in $[0, \infty)$ since $\dim V <\infty$, and by Remark \ref{norm},  $\|y\bo y\| =
\sup \sigma(y\bo y).$
On the other hand, $V$ is a finite dimensional Hilbert space with inner-product
 \begin{equation}\label{<>}
 \langle x,y \rangle =\mathrm{Tr}(x\bo y) \qquad (x,y\in V).
 \end{equation}
For each self-adjoint operator $T$ on $(V, \langle \cdot, \cdot\rangle)$, we use the notation
$$\Lambda(T)=\sup \{ \langle Tz, z\rangle \colon z\in V, \langle z,z\rangle =1\}$$
to denote the maximum eigenvalue of $T$.  In particular, we have
\[
\Lambda(y\bo y) =\sup \sigma(y\bo y),
\]
as $y\bo y$ is a positive self-adjoint operator on the Hilbert space $V$ (cf.\,\cite[Lemma 1.2.22]{book1}).

Since $h_{a_k}(x) =\|x-a_k\|-\|a_k\|\geq -\|x\|$ for all $x\in V$, we have
\begin{equation}\label{lowerbound}
(2r_k)^{-1}(\|(x-a_k)\bo (x-a_k)\|-r_k^2)\geq -2\|x\|
\end{equation}
for sufficiently large $k$ , which will be useful later.

 The start we prove the following technical lemma.
 \begin{lemma}\label{techlem}
 Let $(a_k)$ be a sequence in $V$ such that $r_k=\|a_k\|\to\infty$. Let $a_k = \sum_{i=1}^r \lambda_{ik}e_{ik}$ be a spectral decomposition of $a_k$, with $\lambda_{1k}\geq \lambda_{2k}\geq \ldots\geq \lambda_{rk}\geq 0$ and mutually orthogonal minimal tripotents $e_{1k}, \ldots, e_{rk}$.

If $\alpha_{ik}= r_k-\lambda_{ik}\to \alpha_i\in [0,\infty]$ and $e_{ik}\to e_i$ for all $i$, then
\begin{eqnarray}\label{techeqn}
\lefteqn{\lim_{k\to\infty} (2r_k)^{-1}(\|(x-a_k)\bo (x-a_k)\|-r_k^2)} \nonumber\\
&  & = \sup_{u\in V_2(e_I)\colon \langle u,u\rangle =1}
\langle (-\frac{1}{2}(e_I\bo x+ x\bo e_I)-\sum_{i\in I}\alpha_i (e_i\bo e_i))u,u\rangle,
\end{eqnarray}
where  $I =\{i\colon \alpha_i<\infty\}$ and $e_I =\sum_{i\in I} e_i$.
 \end{lemma}
 \begin{proof}
 We will show that every subsequence of $((2r_k)^{-1}(\|(x-a_k)\bo (x-a_k)\|-r_k^2))_k$ has a convergent subsequence whose limit is the right-hand side of (\ref{techeqn}). Let $((2r_{m})^{-1}(\|(x-a_{m})\bo (x-a_{m})\|-r_{m}^2))_m$ be a subsequence. Note that
 \[
 (2r_{m})^{-1}(\|(x-a_{m})\bo (x-a_{m})\|-r_{m}^2)
  = \Lambda\left(\frac{1}{2r_{m}}x\bo x - \frac{1}{2}(\frac{a_{m}}{r_{m}}\bo x+ x\bo \frac{a_{m}}{r_{m}}) +\frac{1}{2r_{m}}(a_{m}\bo a_{m} -r^2_{m}I)\right),
 \]
where $I\colon V\to V$ is the identity operator.

There exists $w^m\in V$ with $\langle w^m,w^m\rangle =1$ such that
 \begin{multline*}
\Lambda\left(\frac{x\bo x}{2r_{m}} - \frac{1}{2}(\frac{a_{m}}{r_{m}}\bo x+ x\bo \frac{a_{m}}{r_{m}}) +\frac{1}{2r_{m}}(a_{m}\bo a_{m} -r^2_{m}I)\right)
  \\= \left\langle \left(\frac{x\bo x}{2r_{m}}- \frac{1}{2}(\frac{a_{m}}{r_{m}}\bo x+ x\bo \frac{a_{m}}{r_{m}}) +\frac{1}{2r_{m}}(a_{m}\bo a_{m} -r^2_{m}I)\right)w^m,\,w^m\right\rangle.
 \end{multline*}

 After taking a further subsequence, we may assume that $w^m\to w$ and $\lambda_{im}/r_{m}\to\mu_i\in [0,1]$ for all $i$.  So, $a_{m}/r_{m}\to a=\sum_{i=1}^r \mu_i e_i$. Also note that $\mu_i = 1$ for all $i\in I$, as
 \[
 \mu_i =\lim_{m} \frac{\lambda_{im}}{r_{m}} = \lim_{m} \frac{r_{m}-\alpha_{im}}{r_{m}} = 1 \mbox{\qquad for all $i\in I$.}
 \]

 Consider the Peirce decomposition $V =\bigoplus_{0\leq s\leq t\leq r} V_{st}^m$ with respect to $e_{1m},\ldots,e_{rm}$
 and write
$$w^m = \sum_{0\leq s\leq t\leq r} w^m_{st}.$$
 Set $\lambda_{0m}=0$ and $\alpha_{0m} =r_{m}$ for all $m$.
From (\ref{ij}), we have
 \begin{equation*}\label{cases}
 (e_{im}\bo e_{im}) w_{st}^m = \left\{\begin{array}{ll} w_{st}^m & (i=s=t)\\
 \frac{1}{2}w_{st}^m & (i=s\neq t\, {\rm or}\, i=t\neq s)\\
  0 &\mbox{ otherwise.}
  \end{array}\right.
 \end{equation*}

Therefore
 \begin{align}
 \langle \frac{1}{2r_{m}}(a_{m}\bo a_{m} -r^2_{m}I)w^m, \, w^m\rangle &   =
 \langle \sum_{i=1}^r \frac{1}{2r_{m}}(\lambda_{im}^2(e_{im}\bo e_{im})w_m -r^2_{m}w^m),\,w^m\rangle  \notag\\
 &   =
\sum_{0\leq s\leq r} \frac{\lambda_{sm}^2-r_{m}^2}{2r_{m}}\langle w^m_{ss},\,w^m_{ss}\rangle +
\sum_{0\leq s<t\leq r} \frac{(\lambda_{sm}^2+\lambda_{tm}^2)/2-r_{m}^2}{2r_{m}}\langle w^m_{st},\,w^m_{st}\rangle \notag \\
 &  = -\sum_{0\leq s\leq r} \frac{\alpha_{sm}(2r_{m}-\alpha_{sm})}{2r_{m}}\langle w^m_{ss},w^m_{ss}\rangle \notag \\
 & \qquad
-\sum_{0\leq s<t\leq r} \Big( \frac{\alpha_{sm}(2r_{m}-\alpha_{sm})}{4r_{m}}+ \frac{\alpha_{tm}(2r_{m}-\alpha_{tm})}{4r_{m}}\Big)\langle w^m_{st},w^m_{st}\rangle. \label{dotprod}
 \end{align}

Note that, as the of set minimal tripotents is compact and $e_{1k}, \ldots,e_{rk}$ are mutually othogonal tripotents, $e_1,\ldots,e_r$ are mutually orthogonal minimal tripotents. Let $V = \bigoplus_{0\leq s\leq t\leq r} V_{st}$ and  $w = \sum_{0\leq s\leq t\leq r} w_{st}$
be the Peirce decompositions with respect to $e_1,\ldots,e_r$.

 Then  $w_{st}=0$ if $\{s,t\}\not\subset I$. Indeed, $w^m_{st}\to w_{st}$ for all $0\leq s\leq t\leq r$,
 and if $w_{st}\neq 0$ for some $\{s,t\}\not\subset I$, then the right-hand side of (\ref{dotprod}) goes to $-\infty$ as $m\to\infty$, since
 \[
\frac{ \alpha_{sm}(2r_{m}-\alpha_{sm})}{2r_{m}}\geq \frac{ \alpha_{sm}}{2}\to\infty \mbox{\quad or\quad  }  \frac{\alpha_{tm}(2r_{m}-\alpha_{tm})}{2r_{m}}\geq \frac{\alpha_{tm}}{2}\to\infty.
 \]
 As $\langle \frac{x\bo x}{2r_{m}}w^m,w^m\rangle \to 0$ and
 \[
  \langle - \frac{1}{2}(\frac{a_{m}}{r_{m}}\bo x+ x\bo \frac{a_{m}}{r_{m}})w^m,w^m\rangle \to   \langle- \frac{1}{2}(a\bo x+ x\bo a)w,w\rangle,
 \]
 we find that $(2r_k)^{-1}(\|(x-a_k)\bo (x-a_k)\|-r_k^2)\to-\infty$, which contradicts (\ref{lowerbound}).  Hence
 \[
 w\in \bigoplus_{s,t\in I\colon s\leq t} V_{st} = V_2(e_I).
 \]

 Now using the Peirce decomposition $V = V_2(e_I)\oplus V_1(e_I)\oplus V_0(e_I)$ with respect to the tripotent $e_I$ and the
 Peirce multiplication rules,
  we find for each $z\in V_2(e_I)$ (and in particular for $w$) that
 \[
\langle (a\bo x)z,z\rangle  = \langle \{e_I, x,z\}+ \{\sum_{i\not\in I} \mu_i e_I, x,z\},z\rangle = \langle(e_I\bo x)z,z\rangle
 \]
and
 \[
\langle(x\bo a)z,z\rangle = \langle\{x,e_I,z\}+ \{x,\sum_{i\not\in I} \mu_i e_I,z\},z\rangle = \langle(x\bo e_I)z,z\rangle.
 \]

Hence we deduce from (\ref{dotprod}) that
 \begin{align}
 \limsup_{m\to\infty}\, (2r_{m})^{-1}(\|(x-a_m)\bo (x-a_m)\|-r_{m}^2) & \leq  \langle- \frac{1}{2}(a\bo x+ x\bo a)w,w\rangle - \sum_{s\in I}\alpha_s\langle w_{ss},w_{ss}\rangle \notag\\
 & \qquad
 - \sum_{s,t\in I\colon s<t} \frac{1}{2}(\alpha_s+\alpha_t) \langle w_{st},w_{st}\rangle \notag\\
 &  = \langle- \frac{1}{2}(e_I\bo x+ x\bo e_I)w,w\rangle - \sum_{s\in I}\alpha_s\langle (e_s\bo e_s)w,w\rangle.\label{55}
 \end{align}

Next we show
 \begin{equation}\label{liminf}
\liminf_{m\to\infty}\, (2r_{m})^{-1}(\|(x-a_m)\bo (x-a_m)\|-r_{m}^2)  \geq   \langle- \frac{1}{2}(e_I\bo x+ x\bo e_I)w,w\rangle - \sum_{s\in I}\alpha_s\langle (e_s\bo e_s)w,w\rangle.
  \end{equation}

Let $e_I^m=\sum_{i\in I} e_{im}$ and set $u^m = P_2(e_I^m)w^m$. Then $u^m\to P_2(e_I)w = w$  and
 \begin{multline*}
\Lambda\left(\frac{x\bo x}{2r_{m}} - \frac{1}{2}(\frac{a_{m}}{r_{m}}\bo x+ x\bo \frac{a_{m}}{r_{m}}) +\frac{1}{2r_{m}}(a_{m}\bo a_{m} -r^2_{m}I)\right)
  \\ \geq
   \left\langle \left(\frac{x\bo x}{2r_{m}}- \frac{1}{2}(\frac{a_{m}}{r_{m}}\bo x+ x\bo \frac{a_{m}}{r_{m}}) +\frac{1}{2r_{m}}(a_m\bo a_m -r^2_{m}I)\right)u^m,u^m\right\rangle\langle u^m,u^m\rangle^{-1}
 \end{multline*}
 for large $m$, as $\langle u^m,u^m\rangle\to \langle w,w\rangle =1$. Since
 \[
   \langle (\frac{x\bo x}{2r_{m}}- \frac{1}{2}(\frac{a_{m}}{r_{m}}\bo x+ x\bo \frac{a_{m}}{r_{m}}))u^m,u^m\rangle \to
  \langle- \frac{1}{2}(e_I\bo x+ x\bo e_I)w,w\rangle
 \]
 and
 \[
  \lim_{m} \langle \frac{1}{2r_{m}}((a_{m}\bo a_{m}) u^m -r^2_{m}u^m),u^m\rangle =  - \sum_{s\in I}\alpha_s\langle (e_s\bo e_s)w,w\rangle,
 \]
 the inequality follows.

 From (\ref{55}) and  (\ref{liminf}), we now obtain
 \[
\lim_{m} (2r_{m})^{-1}(\|(x-a_{m})\bo (x-a_{m})\|-r_{m}^2) =  \langle- \frac{1}{2}(e_I\bo x+ x\bo e_I)w,w\rangle - \sum_{s\in I}\alpha_s\langle (e_s\bo e_s)w,w\rangle,
 \]
which implies that the left-hand side of (\ref{techeqn}) does not exceed the right-hand side.

 To prove equality of the two sides, pick $v\in V_2(e_I)$ with $\langle v,v\rangle =1$  such that
  \begin{equation*}
\langle (-\frac{1}{2}(e_I\bo x+x\bo e_I)-\sum_{i\in I}\alpha_i (e_i\bo e_i))v, \,v\rangle
= \sup_{u\in V_2(e_I)\colon \langle u,u\rangle =1}
\langle (-\frac{1}{2}(e_I\bo x+x\bo e_I)-\sum_{i\in I}\alpha_i (e_i\bo e_i))u, \,u\rangle.
 \end{equation*}
 Let $v^m = P_2(e^m_I)v$.  Again by definition of $w^m$, we have for large $m$,
  \begin{multline*}
 \langle (\frac{x\bo x}{2r_{m}}- \frac{1}{2}(\frac{a_{m}}{r_{m}}\bo x+ x\bo \frac{a_{m}}{r_{m}}) +\frac{1}{2r_{m}}(a_{m}\bo a_{m} -r^2_{m}I))w^m,w^m\rangle\\
 \geq
 \langle (\frac{x\bo x}{2r_{m}} - \frac{1}{2}(\frac{a_{m}}{r_{m}}\bo x+ x\bo \frac{a_{m}}{r_{m}}) +\frac{1}{2r_{m}}(a_{m}\bo a_{m} -r^2_{m}I))v^m, v^m\rangle \langle v^m, v^m\rangle^{-1},
  \end{multline*}
where $e^m_I=\sum_{i\in I}e_{im}$ and $v^m =P_2(e^m_I)v\to P_2(e_I)v =v$.

Write $v^m  = \displaystyle \sum_{s,t\in I\colon s\leq t} v^m_{st} \in \bigoplus_{0\leq s\leq t\leq r} V_{st}^m$. Then
$\langle \frac{x\bo x}{2r_{m}} v^m,v^m\rangle\to 0$ and
  \[
   \langle - \frac{1}{2}(\frac{a_m}{r_{m}}\bo x+ x\bo \frac{a_m}{r_{m}})v^m,v^m\rangle
   \to \langle - \frac{1}{2}(a\bo x+ x\bo a)v,v\rangle.
    \]

As before,
    \begin{align}
    \langle \frac{1}{2r_{m}}(a_m\bo a_m -r^2_{m}I)v^m,v^m\rangle  & =
-\sum_{s\in I} \frac{\alpha_{sm}(2r_{m}-\alpha_{sm})}{2r_{m}}\langle v^m_{ss}, v^m_{ss}\rangle \notag\\
& \qquad\qquad  -
\sum_{s,t\in I\colon s<t} \Big( \frac{\alpha_{sm}(2r_{m}-\alpha_{sm})}{4r_{m}}+ \frac{\alpha_{tm}(2r_{m}-\alpha_{tm})}{4r_{m}}\Big)\langle v^m_{st},v^m_{st}\rangle \notag\\
&  \to  -\sum_{s\in I} \alpha_s\langle(e_s\bo e_s) v,v\rangle.\notag
 \end{align}

 We conclude that
   \[
  \lim_{m\to\infty} (2r_{m})^{-1}(\|(x-a_{m})\bo (x-a_{m})\|-r_{m}^2)
  \geq
   \langle- \frac{1}{2}(e_I\bo x+ x\bo e_I)v - \sum_{s\in I}\alpha_s (e_s\bo e_s)v,v\rangle,
     \]
which completes the proof.
 \end{proof}

\begin{remark*}\rm In (\ref{techeqn}), we have
   \begin{multline}
\sup_{u\in V_2(e_I)\colon \langle u,u\rangle =1}
\langle (-\frac{1}{2}(e_I\bo x+x\bo e_I)-\sum_{i\in I}\alpha_i (e_i\bo e_i))u,u\rangle
\\
=
 \sup_{u\in V_2(e_I)\colon \langle u,u\rangle =1}
\langle (-\frac{1}{2}(e_I\bo P_2(e_I)x+P_2(e_I)x\bo e_I)-\sum_{i\in I}\alpha_i (e_i\bo e_i))u,u\rangle\\
 =
 \Lambda_{V_2(e_I)}(-\frac{1}{2}(e_I\bo P_2(e_I)x+P_2(e_I)x\bo e_I)-\sum_{i\in I}\alpha_i (e_i\bo e_i)),
    \end{multline}
where the latter denotes the maximum eigenvalue of the operator
 $$-\frac{1}{2}(e_I\bo P_2(e_I)x+P_2(e_I)x\bo e_I)-\sum_{i\in I}\alpha_i (e_i\bo e_i)$$
  restricted to the subspace $V_2(e_I)$, which it leaves invariant.
 \end{remark*}

We are now ready to describe the horofunction boundary $V(\infty)$.
 \begin{theorem}\label{formofh}
 Let $h$ be a horofunction in ${V}(\infty)$. Then there exist $I\subset \{1,\ldots,r\}$ nonempty, mutually orthogonal minimal tripotents $e_i\in V$ and $\alpha_i\geq 0$ for  $i\in I$, with $\min_{i\in I}\alpha_i =0$, such that
\begin{multline}\label{formh2}
h(x) = \sup\{\langle (-\frac{1}{2}(e_I\bo x+x\bo e_I)-\sum_{i\in I}\alpha_i (e_i\bo e_i))u,u\rangle\colon
 u\in V_2(e_I),\, \langle u,u\rangle =1\} \\
= \Lambda_{V_2(e_I)}(-\frac{1}{2}(e_I\bo P_2(e_I)x+P_2(e_I)x\bo e_I)-\sum_{i\in I}\alpha_i (e_i\bo e_i))
\qquad (x\in V),
\end{multline}
where $e_I =\sum_{i\in I} e_i$ is a tripotent with Peirce $2$-space $V_2(e_I)$ and Peirce projection $P_2(e_I)$.

Conversely, each function $h \colon V\to \mathbb{R}$ of the form in (\ref{formh2}) is a horofunction of $V$.
\end{theorem}
\begin{proof}
Let $h =\lim_k h_{a_k}$ be a horofunction, where $a_k\in V$ and  $r_k= \|a_k\|\to\infty$.   For each $k$, let
$ a_k = \sum_{i=1}^r \lambda_{ik}e_{ik}$ be a spectral decomposition, with
pairwise  orthogonal minimal tripotents $e_{ik}$ and $\|a_k\|=\lambda_{1k}\geq \ldots\geq \lambda_{rk}\geq 0$.

After taking a subsequence, we may assume  $\alpha_{ik} = r_k -\lambda_{ik}\to \alpha_i \in [0,\infty]$ and $e_{ik}\to e_i$ for all $i =1,\ldots,r$, where $e_1,\ldots,e_r$ are mutually orthogonal minimal tripotents. Let $I=\{i\colon \alpha_i<\infty\}$ and $e_I =\sum_{i\in I} e_i$, and note that  $I\neq \emptyset$, since $\alpha_1=0$.

We have
\[
h_{a_k}(x) = \frac{\|x-a_k\|^2-\|a_k\|^2}{\|x-a_k\|+\|a_k\|} =  \frac{(2r_k)^{-1}(\|(x-a_k)\bo (x-a_k)\|-r_k^2)}{2^{-1}(\|r_k^{-1}(x-a_k)\|+1)}
\]
and $2^{-1}(\|r_k^{-1}(x-a_k)\|+1)\to 1$. Hence  (\ref{formh2}) follows readily from Lemma \ref{techlem}
and the preceding remark.

Conversely, let $h \colon V\to\mathbb{R}$ be of the form (\ref{formh2}), where $I$ is a nonempty subset of
$\{1, \ldots, r\}$, $\min_{i\in I}\alpha_i =0$ and $e_I=\sum_{i\in I} e_i$ is the sum of  mutually orthogonal minimal tripotents $e_i$.

We show $h \in V(\infty)$. For $k= 1, 2, \ldots,$
 define
$$a_k= ke_I- \sum_{i\in I} \alpha_ie_i.$$

For $k\geq \max_{i\in I}\alpha_i$, we have  $r_k = \|a_k\| =k$ and, in the notation of Lemma \ref{techlem},
  $$\alpha_{ik} =\left \{\begin{matrix} \alpha_i & (i\in I)\\  k & (\rm{otherwise}) \end{matrix}\right.$$
It follows from  Lemma \ref{techlem} that
\[
\lim_{k\to\infty} (2r_k)^{-1}(\|(x-a_k)\bo (x-a_k)\|-r_k^2)
= \sup_{u\in V_2(e_I)\colon \langle u,u\rangle =1}
\langle (-\frac{1}{2}(e_I\bo x+x\bo e_I)-\sum_{i\in I}\alpha_i (e_i\bo e_i))u,u\rangle
\]
and
\begin{multline*}
\lim_{k\to\infty} h_{a_k}(x)
  =  \lim_{k\to\infty} \frac{(2r_k)^{-1}(\|(x-a_k)\bo (x-a_k)\| - r_k^2)}{2^{-1}(\|r_k^{-1}(x-a_k)\|+1)}\\
 =   \sup_{u\in V_2(e_I)\colon \langle u,u\rangle =1}
\langle (-\frac{1}{2}(e_I\bo x+x\bo e_I)-\sum_{i\in I}\alpha_i (e_i\bo e_i))u,u\rangle =h(x)
\end{multline*}
for all $x\in V$. Hence $h\in V(\infty)$.
\end{proof}

If $I$ is a singleton in the preceding theorem, $e_I =e$, where $e$ is a minimal tripotent. In that case the last term of (\ref{formh2}) vanishes
 and $h$ is a real continuous linear functional of $V$. Indeed, $P_2(e)(V) = V_2(e) = \mathbb{C}e$
 implies $P_2(e)x =\ell(x)e$ for some functional $\ell\in V^*$ and
 \begin{multline*}
h(x)  =  \langle -\frac{1}{2}(e\bo P_2(e)x + P_2(e)x\bo e)e,e\rangle \langle e,e\rangle^{-1}=   -\frac{1}{2}(\langle e, (\ell(x)e\bo e) e\rangle  + \langle(\ell(x)e\bo e)e,e\rangle) \langle e,e\rangle^{-1}\\
 =   -\frac{1}{2}(\ol{\ell(x)}+\ell(x))  = -\re\, \ell(x).
\end{multline*}

\begin{remark}\label{rm}
 In the course of proving the preceding theorem, we observe that each horofunction $h\in V(\infty)$ can actually  be constructed from a sequence $(a_k)$ going to infinity along a straight line, which is a geodesic in the normed space $V$. In fact, the sequence $a_k= ke_I -\sum_{i\in I}\alpha_i e_i$ used in the proof lies on the straight line, $t\mapsto te_I -\sum_{i\in I}\alpha_ie_i$ in the flat $\oplus_{i\in I}\mathbb{R}e_i$. Also note that if $k\geq m$ with $k\geq \max_{i\in I}\alpha_i$, then $\|a_k\| =k$ and  $h_{a_k}(a_m) = \|a_k-a_m\| -\|a_k\| = (k-m) - k = -m$,
 so that  $h(a_m) =-m$ for all $m$.
 \end{remark}
 By the remark we have the following corollary.
\begin{corollary}\label{buse}  Each horofunction in ${V}(\infty)$ is a Busemann point.
\end{corollary}

For general  finite dimensional normed vector spaces  it need not be true that all horofunctions are Busemann points, see \cite{Wa2}.

\section{Homeomorphism  onto the dual unit ball}\label{homdualball}

In this section we give a homeomorphism of the metric compactification of $V$ onto the closed dual unit ball $B^*$ of $(V,\|\cdot\|)$.  We subsequently show in the next section that this homeomorphism maps each equivalence class in $V(\infty)/\sim$ onto the relative interior of a boundary face of $B^*$. So, the dual ball $B^*$ captures the geometry of the metric compactification of $V$.


To prove these results we need a lemma concerning the partial ordering on the set of tripotents of a JB$^*$-triple. Recall that, given two tripotents $c$ and $e$ in $V$, we shall write $c \leq e$ if $e-c$ is a tripotent in $V$ orthogonal to $c$, or equivalently, if and only if $P_2(c)e =c$ (\cite[Corollary 1.7]{FriRu85}). We also have $c\leq e$ if and only if $c$ is a projection in the JB$^*$-algebra $V_2(e)$. Particularly,
\begin{equation}\label{eq ordered trip}
	c=\{e,c,e\}
\end{equation}
(cf.\,\cite[pp.34-36]{book1}).
\begin{lemma}\label{eicj5} Given two tripotents $c, e\in V$, the following conditions are equivalent:
	\begin{enumerate}
		\item[\rm (i)] $c\leq e$,
		\item[\rm (ii)] $\{c, P_2(c)e,c\} + \{c,c, P_2(c)e\} =2c$,
		\item[\rm (iii)] $\{c, e,c\} + \{c,c, e\} =2c$,
		\item[\rm (iv)] $B(e,c)w=0$ for all $w\in V_2(c)$.
	\end{enumerate}
\end{lemma}
\begin{proof}
	(i) $\Rightarrow$ (ii). If $c\leq e$, then $P_2(c)e =c$ and  (ii) follows directly.

	(ii) $\Rightarrow$ (i). We note that  $ \{c, P_2(c)e,c\} \in V_2(c)$ by \eqref{eq Peirce arit 1},	and $\{ c,c,P_2(c)e \}=P_2(c)e\in V_2(c)$.
	
	Observe that $c$ is a maximal tripotent in the JB*-triple $V_2(c)$, and hence $c$ is an extreme point of the closed unit ball of $V_2(c)$ (\cite[Theorem 3.2.3]{book1}). As $\|\{c, P_2(c)e,c\}\|, \|P_2(c)e\| \leq 1$, we conclude that  $P_2(c)e=c$, that is, $c \leq e$.
	
	(i) $\Rightarrow$ (iii). Suppose $c\leq e$. The orthogonality between $e$ and $e-c$ gives clearly (iii).
	
	(iii) $\Rightarrow$ (i). The arguments used in the implication (ii) $\Rightarrow$ (i) apply,
 by observing that $ \{c, e,c\}= \{c, P_2(c)e,c\} \in V_2(c)$  from \eqref{eq Peirce arit 1} and \eqref{eq Peirce arit 2}, which gives  $\{c,c,e\} =2c-  \{c, e,c\} \in V_2(c)$.
	
	(i) $\Rightarrow$ (iv). Let $w\in V_2(c)$.  Using orthogonality, \eqref{eq ordered trip} and the Jordan triple identity, we deduce
	\begin{eqnarray*}
		 B(e,c)w & = & w - 2 \{ e, c, w \} + \{ e, \{ c, w, c \} ,e \}\\
		& = & w- 2 \{ e, c, w \} + 2\{ e, c, \{ e, c, w \}\} - \{ w,c, \{e,c,e\} \}\\
		&= &  w- 2 \{ e, c, w \} + 2\{ e, c, \{ e, c, w \}\} - \{ w,c, c \}\\
		&= & w- 2 \{ e, c, w \} + 2\{ e, c, \{ e, c, w \}\} - w \\
		& = & w- 2 w + 2w - w  = 0.
	\end{eqnarray*}
	
	(iv) $\Rightarrow$ (i).
	Let $a = P_2(c)e \in V_2(c)$. We show $a=c$. Using the identities (JP1), (JP3) and (JP2) in \cite[Appendix]{loos},
	one deduces that
	$$ Q(c)B(a,c)w 
	= Q(c)B(e,c)w =0, \qquad  (w\in V_2(c)).$$
	Now take $w = c-\{c,a,c\}\in V_2(c)$. Then we have
	\begin{eqnarray*}
		B(a,c)w &=& B(a,c)c - B(a,c)\{c,a,c\}\\
		&= & c - 2a +\{a,c,a\} - (\{c,a,c\} - 2\{a,c, \{c,a,c \}\} + \{a, \{c,\{c,a,c\},c\},a\})\\
		&=& c - 2a +\{a,c,a\} - \{c,a,c\}+ 2\{c,a,a\} -\{a,a,a\}\\
		&=& \{c-a, c-a,c-a\},
	\end{eqnarray*}
	where
	$\{c,a,a\} = \{c,a,\{c,c,a\}\} = \{\{c,a,c\},c,a\} - \{c, \{a,c,c\},a\} + \{c,c,\{c,a,a\}\}= \{\{c,a,c\},c,a\}$.
	As  $\{c-a, c-a,c-a\}\in V_2(c)$, it  follows from (\ref{x*norm}) that
	$$\|\{c-a,c-a,c-a\}\| = \|Q(c)\{c-a,c-a,c-a\}\|=\|Q(c)B(a,c)w\|=0,$$
	so that $c-a=0$ by (\ref{c*-identity}).
	\end{proof}

\begin{remark}\label{+op}\rm
Given tripotents $c,e\in V$ and identity map $I\colon V \rightarrow V$, we have  $\|c\bo P_2(c)e + P_2(c)e \bo c\|\leq 2$ and
$2I- (c\bo P_2(c)e + P_2(c)e \bo c)$ is a positive operator on the Hilbert space $(V_2(c), \langle \cdot,\cdot\rangle)$ by
Remark \ref{norm}. Hence $\langle (2I-(c\bo P_2(c)e + P_2(c)e \bo c))v, v\rangle \geq 0$  for all $v\in V_2(c)$. Moreover,
$\langle (2I-(c\bo P_2(c)e + P_2(c)e \bo c))v, v\rangle =0 $ for all $v\in V_2(c)$ is equivalent to $2I=c\bo P_2(c)e + P_2(c)e \bo c$ on $V_2(c)$.
\end{remark}
The previous lemma has the following useful corollary.
\begin{corollary}\label{h=h'}
Suppose that $h$ and $h'$ are horofunctions given, respectively, by (\ref{formh2}) and
$$h'(x) = \Lambda_{V_2(c_J)}(-\frac{1}{2}(c_J\bo P_2(c_J)x+P_2(c_J)x\bo c_J)-\sum_{j\in J}\beta_j (c_j\bo c_j)).$$
Let $a=\sum_{i\in I}\alpha_i e_i$ and  $b=\sum_{j\in J} \beta_jc_j$. Then $h=h'$ if and only if $e_I=c_J$ and $a=b$.
\end{corollary}
\begin{proof}
The sufficiency follows from the observation that $e_I=c_J$ and $a=b$ implies that
$$\sum_{i\in I}\alpha_i (e_i\bo e_i)
= (\sum_{i\in I}\alpha_i e_i) \bo e_I = ( \sum_{j\in J} \beta_jc_j)\bo c_J =\sum_{j\in J}\beta_j (c_j\bo c_j).$$

Conversely, let $\lim_k h_{a_k}=h=h' = \lim_k h_{b_k}$,
where we can choose $a_k= ke_I -a$ and $b_k= kc_J -b$ by Remark \ref{rm}.
As  $h(a_k)=-k$, we have $k + h'(a_k)=0$ and
   \begin{multline*}
 k + h'(a_k)   = k + \sup_{v\in V_2(c_J)\colon \langle v,v\rangle =1}\langle (-\frac{1}{2}(c_J\bo (ke_I-a)+(ke_I-a)\bo c_J -\sum_{j\in J}\beta_j (c_j\bo c_j))v,v\rangle\\
= k +  \sup_{v\in V_2(c_J)\colon \langle v,v\rangle =1}( \langle -\frac{k}{2}((c_J\bo P_2(c_J)e_I) +(P_2(c_J)e_I \bo c_J))v,v\rangle \\
+ \langle (\frac{1}{2}(c_J \bo a + a \bo c_J) -\sum_{j\in J}\beta_j(c_j\bo c_j) )v,v\rangle )\\
= k( \sup_{v\in V_2(c_J)\colon \langle v,v\rangle =1} ( \langle (I -\frac{1}{2}((c_J\bo P_2(c_J)e_I) +(P_2(c_J)e_I \bo c_J)))v, v\rangle\\
+ \langle (\frac{1}{2k}(c_J \bo a + a \bo c_J) - \frac{1}{k}\sum_{j\in J}\beta_j (c_j \bo c_j)) v,v\rangle))
   \end{multline*}
 for $k=1, 2, \ldots $. So we find that $\langle (I - \frac{1}{2}(c_J\bo P_2(c_J)e_I + P_2(c_J)e_I \bo c_J))v,v\rangle=0$ for all $v\in V_2(c_J)$, and hence $2I =c_J\bo P_2(c_J)e_I + P_2(c_J)e_I \bo c_J$ by Remark \ref{+op} on $V_2(c_J)$. Moreover,
\begin{equation}\label{ab}
 \sup_{v\in V_2(c_J)\colon \langle v,v\rangle =1}
 \langle(\frac{1}{2}(c_J \bo a + a \bo c_J) - \sum_{j\in J}\beta_j (c_j \bo c_j)) v,v\rangle=0.
\end{equation}
In particular, we have $2c_J= (c_J\bo P_2(c_J)e_I)(c_J) +(P_2(c_J)e_I \bo c_J)(c_J)$, so $c_J \leq e_I$ by Lemma \ref{eicj5}.

Analogously, $0= k + h'(b_k) = k + h(b_k)$ implies  $e_I \leq c_J$ and
\begin{equation}\label{ba}
 \sup_{v\in V_2(e_I)\colon \langle v,v\rangle =1}
 \langle(\frac{1}{2}(e_I \bo b + b \bo e_I) -\sum_{j\in J}\alpha_i (e_i \bo e_i)) v,v\rangle =0.
\end{equation}

We conclude that $e_I=c_J$ and note that this implies that
\[
\frac{1}{2}(c_J \bo a + a \bo c_J) =\frac{1}{2}(e_I \bo a + a \bo e_I) = \sum_{i\in I}\alpha_i (e_i\bo e_i)
\]
and
\[
\frac{1}{2}(e_I \bo b + b \bo e_I) = \frac{1}{2}(c_J \bo b + b \bo c_J) =  \sum_{j\in J}\beta_j (c_j \bo c_j).
\]
It now follows from (\ref{ab}) and (\ref{ba}) that $ \sum_{i\in I}\alpha_i (e_i \bo e_i)= \sum_{j\in J}\beta_j (c_j\bo c_j)$ on $V_2(e_I)=V_2(c_J)$, as both operators are self-adjoint. In particular,
$$a=\sum_{i\in I}\alpha_i e_i = \sum_{i\in I}\alpha_i (e_i \bo e_i)(e_I) =  \sum_{j\in J}\beta_j (c_j\bo c_j)(c_J)= \sum_{j\in J}\beta_j c_j=b.$$
\end{proof}

To show  that the metric compactification $V\cup V(\infty)$ is homeomorphic to  closed dual unit ball of $(V,\|\cdot\|)$, we identify $V$ with its (algebraic) dual space $V^*$ by using an inner-product on $V$, which we can do as $V$ is finite dimensional.  For convenience we adjust the inner-product $\langle\cdot,\cdot\rangle$ on $V$ to define a new inner-product $[\cdot,\cdot]$ such that $[c,c]=1$ for each minimal tripotent $c\in V$.

To realise this we note that a finite dimensional $\mathrm{JB}^*$-algebra $V$ decomposes into a finite $\ell_\infty$-sum,
$V = V_1\oplus\cdots\oplus V_d$, of {\it Cartan factors}, each of which contains a minimal tripotent (cf.\,\cite[Theorem 3.3.5, Theorem 3.8.17]{book2}. In a finite dimensional Cartan factor $V_j$, the tripotents form a compact submanifold $M_j$ of $V_j$,  in which
 the minimal tripotents form a connected component $N_j$ \cite[\S 5]{loos}. As shown in
 \cite[Proposition 2.2]{isi}, the tangent space $T_e(N_j)$ at each $e\in N_j$ identifies
 with $\mathrm{i}A(e) \oplus V_1(e)$, where $A(e)$ is defined in (\ref{ae}).

 In particular, $\dim (\mathrm{i}A(e) \oplus V_1(e))= 1 + \dim V_1(e)$
 is constant for all $e\in N_j$. On $V_j$ we define, for a fixed $e\in N_j$, a normalised trace form
 $$\langle x,y\rangle_j = \frac{1}{1+ \dim V_1(e)/2} {\rm Trace}\, (x \bo y)\qquad (x,y\in V_j),$$
so that $\langle c,c\rangle_j =1 $ for all $c\in N_j$, and $(V_j, \langle \cdot, \cdot\rangle_j)$ is a Hilbert space.

Henceforth, we denote by $[\cdot,\cdot]$ the inner-product of the Hilbert space direct sum
$V = V_1\oplus\cdots\oplus V_d$, with $(V_j,\langle\cdot,\cdot\rangle_j)$.
Then each minimal tripotent $c\in V$ lies in some $V_j$ and $[c,c] =1$. Moreover, the inner-product
is  associative, i.e.,  $$[\{a,b,y\}, z] = [y, \{b,a,z\}] \qquad (a,b,y, z \in V)$$
(cf.\,\cite[(2.31)]{book1}).
In particular, we have $[P_2(c)y,z] = [y, P_2(c)z]$ for each tripotent $c\in V$. We note that if $a,b\in V$
are triple orthogonal, i.e., $a\bo b=0$, then $[a,b]=0$.

By the Riesz representation theorem, the map
\begin{equation}\label{tilde}
x\in V \mapsto \widetilde x=[ \cdot, x] \in V^*
\end{equation}
is a conjugate linear isomorphism.

Let $D^*=\{\widetilde{x}\in V^*\colon \|x\|_*\leq 1\}$ be the closed dual unit ball in dual space of the $\mathrm{JB}^*$-triple $(V,\|\cdot\|)$.
Then for each $\widetilde{x} =[\cdot,x]\in V^*$, where $x$ has spectral decomposition $x = \sum_{j=1}^r \lambda_j c_j \in V$, we have $\|\widetilde x\|_*=\sum_{j=1}^r \lambda_j$. Indeed,
$$\|\widetilde x\|_* =  \sup_{\|y\|=1} |[ y, x]| = \sup_{\|y\|=1} \sum_{j=1}^r \lambda_j |[ y, c_j]|
= \sup_{\|y\|=1} \sum_{j=1}^r \lambda_j |[ P_2(c_j)y, c_j]| \leq
\sup_{\|y\|=1} \sum_{j=1}^r \lambda_j [ c_j, c_j] = \sum_{j=1}^r \lambda_j,$$
as  $P_2(c_j)y =\mu c_j$ for some $\mu\in\mathbb{C}$ and $|\mu|=\|P_2(c_j)y\| \leq \|y\|=1$. On the other hand, $y=\sum_{j=1}^r c_j$  satisfies $\|y\|\leq 1$ and hence
 $\|\widetilde{x}\|_*\geq | \sum_{j=1}^r [c_j, x]| = \sum_{j=1}^r \lambda_j$.

We define
\begin{equation}\label{D0}
D^\circ=\{ x\in V\colon x=\sum_{j=1}^r \lambda_jc_j\mbox{ spectral decomposition with }\sum_{j=1}^r \lambda_j\leq 1\}.
\end{equation}
It follows from the previous observations that the conjugate linear isomorphism $x\in V\mapsto \widetilde{x} \in V^*$ in (\ref{tilde}) maps $D^\circ$ onto $D^*$, and hence $D^\circ$ is the closed unit ball of a norm on $V$.

Thus, to prove that $V\cup V(\infty)$ is homeomorphic to $D^*$, it suffices to show that there exists a homeomorphism $\phi$ from $V\cup V(\infty)$ onto $D^\circ$, which is what we will do.

\begin{remark}\label{uuni}\rm
Given $x\in V$ with spectral decomposition $x=\sum_{i=1}^r \lambda_ie_i$, so $r$ is the rank of $V$, we have that the eigenvalues $\lambda_1\geq \ldots\geq \lambda_r\geq 0$ are unique, but the pairwise orthogonal minimal tripotents $e_i$ need not be unique. We can however collect terms with equal non-zero eigenvalue in the sum and write $x = \sum_{i=1}^s\mu_ic_i$, where $\mu_1>\ldots>\mu_s>0$ and the $c_i$'s are (not necessarily minimal) pairwise orthogonal tripotents. In this case both the $\mu_i$'s and $c_i$'s are unique, see \cite[Corollary 3.12]{loos}. For clarity we refer to this decomposition of $x$, as the {\em unique spectral decomposition}.
\end{remark}

We define the map $\phi\colon V \cup V(\infty) \to {D}^\circ$ by
\begin{equation}\label{phi1}
\phi(x) = \frac{\sum_{i=1}^r (\exp \lambda_i - \exp(-\lambda_i)) c_i}{\sum_{i=1}^r (\exp\lambda_i+\exp(-\lambda_i))}
\end{equation}
for each $x\in V$ with spectral decomposition $x = \sum_{i=1}^r \lambda_ic_i  \in V$, and
\begin{equation}\label{phi2}
\phi(h)  = \frac{\sum_{i\in I} \exp(-\alpha_i) e_i}{\sum_{i\in I} \exp(-\alpha_i)}
 \end{equation}
 for $h\in{V}(\infty)$ of the form,
 \[
 h(x) =  \Lambda_{V_2(e_I)}(-\frac{1}{2}(e_I\bo P_2(e_I)x+P_2(e_I)x\bo e_I)-\sum_{i\in I}\alpha_i (e_i\bo e_i))
 \qquad (x\in V).
 \]
 To see that $\phi(x)$ is well-defined note that the right-hand side of (\ref{phi1}) is the spectral decomposition of $\phi(x)$. Moreover,  if $\lambda_i=0$, then the correspond coefficient in $\phi(x)$ is also $0$. So by switching to the unique spectral decomposition we find that $\phi(x)$ is well-defined by Remark \ref{uuni}. Also $\phi(h)$ is well defined. Indeed, if $h$ was expressed as
 \[
 h(x) =  \Lambda_{V_2(c_J)}(-\frac{1}{2}(c_J\bo P_2(c_J)x+P_2(c_J)x\bo c_J)-\sum_{j\in J}\beta_j (c_j\bo c_j))
 \qquad (x\in V),
 \]
 then $e_I =c_J$ and $a=\sum_{i\in I}\alpha_i e_i=\sum_{j\in J} \beta_jc_j =b$ by Corollary \ref{h=h'}. Viewing $a$ and $b$ in the $\mathrm{JB}^*$-subtriple $V_2(e_I)=V_2(c_J)$ we may assume after relabelling that $I=J$ and $\alpha_i =\beta_i$ for all $i\in I=J$. Now using the unique spectral decomposition and the fact that $e_I =c_J$, we find that $\phi(h)$ is well-defined.

\begin{theorem}\label{homeo} The map $\phi\colon V \cup V(\infty) \to {D}^\circ$ is a homeomorphism.
\end{theorem}

The proof of the theorem will be split in to several lemmas. Note that the interior of $D^\circ$, denoted $\mathrm{int}\, {D}^\circ$, consists of those $x=\sum_{j=1}^r \lambda_jc_j$ with $\sum_{j=1}^r \lambda_j<1$, and the boundary, $\partial D^\circ$, of $D^\circ$ are precisely those $x=\sum_{j=1}^r \lambda_jc_j$ with $\sum_{j=1}^r \lambda_j=1$. This follows from the fact that $x\mapsto \widetilde{x}=[\cdot,x]$ is a conjugate linear isomorphism mapping $D^\circ$ onto $B^*$.
\begin{lemma} We have $\phi (V) \subset \mathrm{int}\, {D}^\circ$ and  $\phi(V(\infty)) \subset \partial {D}^\circ$.
\end{lemma}
\begin{proof} The right-hand side of (\ref{phi1})  is the spectral decomposition of $\phi(x)$ and  \[0\leq (\sum_{i=1}^r e^{\lambda_i}+e^{-\lambda_i})^{-1}(\sum_{i=1}^r e^{\lambda_i}-e^{-\lambda_i})<1,\] so $\phi(x)\in\mathrm{int}\, D^\circ$. Clearly,
$\phi(h)\in \partial D^\circ$.
\end{proof}

\begin{lemma}\label{cty} The map  $\phi\colon V \cup V(\infty) \to {D}^\circ$ is continuous on $V$.
\end{lemma}
\begin{proof}
Let $(v_k)$ be a sequence in $V$ converging to $v\in V$.
To show $\phi(v_k) \rightarrow \phi(v)$ as $k \rightarrow \infty$, we show that each subsequence of $(v_k)$ contains a subsequence $(v_n)$ satisfying $\phi(v_n) \rightarrow \phi(v)$ as $n \rightarrow \infty$.

For each $k$, let $v_k =\sum_{i=1}^r \mu_{ik}c_{ik}$ be a spectral decomposition. By convergence, $(v_k)$ is bounded. Hence each subsequence of $(v_k)$ contains a subsequence
$$v_n = \sum_{i=1}^r \mu_{in}c_{in}$$
such that $\mu_{in} \rightarrow \mu_i\geq 0$ and $c_{in} \rightarrow c_i$ as $n\rightarrow \infty$,
where $c_1, \ldots, c_r$ are mutually orthogonal  minimal tripotents.

It follows that $v= \lim_n v_n =  \sum_{i=1}^r\mu_i c_i$ and
$$\phi(v_n) =  \frac{\sum_{i=1}^r (e^{\mu_{in}}-e^{-\mu_{in}}) c_{in}}{\sum_{i=1}^r e^{\mu_{in}}+e^{-\mu_{in}}}~
\longrightarrow ~ \frac{\sum_{i=1}^r (e^{\mu_i}-e^{-\mu_i}) c_i}{\sum_{i=1}^r e^{\mu_i}+e^{-\mu_i}} =\phi(v)
$$
as $n \rightarrow \infty$, which completes the proof.
\end{proof}

\begin{lemma}\label{=} We have $\phi(V) =\mathrm{int}\, {D}^\circ$.
\end{lemma}
\begin{proof} As $\phi$ is continuous on $V$ and maps $V$ into $\mathrm{int}\, {D}^\circ$, we know from the Brouwer invariance of domain theorem that $\phi(V)$ is open in $\mathrm{int}\, {D}^\circ$. Suppose, for the sake of contradiction, that $\phi(V)\neq \mathrm{int}\, {D}^*$. Then there exists $w\in\partial \phi(V)\cap \mathrm{int}\, {D}^\circ$. Let $(v_n)$ in $V$ be such that $\phi(v_n)\to w$.  As $\phi$ is continuous on $V$, we must have that $\|v^n\|\to\infty$.

Consider the spectral decomposition $v_n = \sum_{i=1}^r \lambda_{in}c_{in}$. After taking a subsequence we may assume that (1) $c_{in}\to c_i$, and (2) $\alpha_{in}=\lambda_{1n}-\lambda_{in}\to \alpha_i \in [0,\infty]$, for all $i=1,\ldots,r$.
Note that $\lambda_{1n} =\|v_n\|\to\infty$.

Let $I=\{ i\colon \alpha_i<\infty\}$. Then $1\in I$ and
\[
\phi(v_n) = \frac{\sum_{i=1}^r (e^{\lambda_{in}}-e^{-\lambda_{in}}) c_{in}}{\sum_{i=1}^r e^{\lambda_{in}}+e^{-\lambda_{in}}}
   = \frac{\sum_{i=1}^r (e^{-\alpha_{in}}-e^{-\lambda_{1n}-\lambda_{in}}) c_{in}}{\sum_{i=1}^r e^{-\alpha_{in}}+e^{-\lambda_{1n}-\lambda_{in}}} \to
  \frac{ \sum_{i\in I} e^{-\alpha_i}c_i}{\sum_{i\in I} e^{-\alpha_i}}.
\]
This implies that $w=\lim_{n\to\infty}\phi(v_n) \in\partial D^\circ$, which contradicts the assumption that $w\in \mathrm{int}\, {D}^\circ$.
\end{proof}

\begin{lemma}\label{sur2} The map $\phi$ satisfies $\phi(V(\infty)) =\partial {D}^\circ$.
\end{lemma}
\begin{proof}
Let $x\in \partial {D}^\circ$. Then $x$ has spectral decomposition $x=\sum_{i=1}^r\lambda_i e_i$ with $\sum_{i=1}^r \lambda_i =1$ and there exists  a $p\in\{1,\ldots,r\}$ such that $\lambda_1\geq \ldots\geq \lambda_p>0$ and $\lambda_{s}=0$ for $p<s\leq r$.

For $i=1,\ldots,p$ put $\mu_i = -\log \lambda_i $ and $\alpha_i = \mu_i -\mu_1$. Then $\alpha_i\geq 0$ and  $\alpha_1=0$.
Now consider the horofunction $h\in V(\infty)$ given by
\[
h(x) =  \Lambda_{V_2(e)}(-\frac{1}{2}(e\bo P_2(e)x+P_2(e)x\bo e)-\sum_{i=1}^p\alpha_i (e_i\bo e_i)) \qquad (x\in V).
 \]
 Then we have
 \[
 \phi(h)  =  \frac{\sum_{i=1}^p \exp(-\alpha_i) e_i}{\sum_{i=1}^p \exp(-\alpha_i)}
  =  \frac{\sum_{i=1}^p \lambda_i  e_i}{\sum_{i=1}^p \lambda_i} = x.
 \]
\end{proof}

\begin{lemma}\label{inj} The map $\phi$ is injective on $V\cup V(\infty)$.
\end{lemma}
\begin{proof}
Suppose  first that $\phi(x) =\phi(y)$. Let $x = \sum_{i=1}^r \lambda_i c_i$ and $y = \sum_{i=1}^r \mu_id_i$ be the spectral decompositions.
 Then we have
\[
\phi(x) = \frac{\sum_{i=1}^r (e^{\lambda_i}-e^{-\lambda_i})  c_i}{\sum_{i=1}^r e^{\lambda_i}+e^{-\lambda_i}} = \frac{\sum_{i=1}^r (e^{\mu_i}-e^{-\mu_i}) d_i}{\sum_{i=1}^r e^{\mu_i}+e^{-\mu_i}} =\phi(y)
\]
where the coefficients of the minimal tripotents of both sides are decreasing through the order of the indices.
If we let for $j=1, \ldots r$,
$$\alpha_j =  (\sum_{i=1}^r e^{\lambda_i}+e^{-\lambda_i})^{-1}(e^{\lambda_j}-e^{-\lambda_j})\mbox{\quad and\quad } \beta_j= (\sum_{i=1}^r e^{\mu_i}+e^{-\mu_i})^{-1}(e^{\mu_j}-e^{-\mu_j}),$$
then   $\alpha_j = \beta_j$ by Remark \ref{uuni}.  It now follows from \cite[Lemma 3.7]{LP} that $\lambda_j=\mu_j$ for $j=1,\ldots,r$.

Note that $\alpha_i=0$ if and only if $\lambda_i=0$, and similarly $\beta_i=0$ if and only if $\mu_i=0$. So by considering the unique spectral decompositions of $\phi(x)$ and $\phi(y)$ and using Remark \ref{uuni} we find that $x=y$.

Now suppose $h, h'\in V(\infty)$ with $\phi(h)=\phi(h')$. Let $h$ be of the form  (\ref{formh2}) and $h'$ of the form
\begin{equation*}
h'(x) =\Lambda_{V_2(c_J)}(-\frac{1}{2}(c_J\bo P_2(c_J)x+P_2(c_J)x\bo c_J)-\sum_{j\in J}\beta_j (c_j\bo c_j)).
\end{equation*}
Then we have
\[
\frac{\sum_{i\in I} \exp(-\alpha_i)e_i}{\sum_{i\in I} \exp(-\alpha_i)} = \frac{\sum_{j\in J} \exp(-\beta_j)c_j}{\sum_{j\in J} \exp(-\beta_j)}
\]
where the coefficients of the minimal tripotents on both sides are strictly positive.
By relabelling the indices, we may assume that $I= \{1,\ldots,p\}$ and $0=\alpha_1\leq\alpha_2\leq\ldots\leq\alpha_p$. Likewise we can assume that $J=\{1,\ldots,q\}$ and  $0=\beta_1\leq\beta_2\leq\ldots\leq\beta_q$. Since the norm of both sides above  in $V$ are equal,
we have
$$\sum_{i\in I} \exp(-\alpha_i) = \sum_{j\in J} \exp(-\beta_j).$$

By Remark \ref{uuni}, we have $p=q$, $e^{-\alpha_i}=e^{-\beta_i}$,
$e_I=e_1 + \cdots + e_p=c_1+ \cdots + c_q =c_J$
and
$\sum_i \alpha_ie_i = \sum \beta_i c_i.$
Hence $h=h'$ by Corollary \ref{h=h'}.
\end{proof}

\begin{lemma} If $(a_k)$ in $V$ is such that $h_{a_k}\to h\in V(\infty)$, then $\phi(a_k) \to \phi(h)$.
\end{lemma}
\begin{proof}
Let $h$ be given by  (\ref{formh2}).
To show that  $\phi(a_k) \to \phi(h)$, we show that each subsequence of  $(\phi(a_k))$ has a convergent subsequence with limit $\phi(h)$.
So let $(\phi(a_m))$ be a subsequence.  Using the spectral decomposition we write $a_m =\sum_{i=1}^r \mu_{im}c_{im}$ with $\mu_{1m}\geq \ldots\geq \mu_{rm}\geq 0$.  As $h$ is a horofunction, $\mu_{1m}=\|a_{m}\|\to\infty$ by Lemma \ref{Rieffel}.

After taking a subsequence we may assume that $\beta_{im} = \mu_{1m}-\mu_{im}\to \beta_i\in [0,\infty]$ and $c_{im}\to c_i$ for all $i$.
Let $J=\{i\colon \beta_i<\infty\}=\{1,\ldots, q\}$ and note that $q\geq1$, as $\beta_1=0$.

It follows from Lemma \ref{techlem} that $h_{a_m}\to h'$ where
\[
h'(x) =\Lambda_{V_2(c_J)}(-\frac{1}{2}(c_J\bo P_2(c_J)x+P_2(c_J)x\bo c_J)-\sum_{j\in J}\beta_j (c_j\bo c_j))\qquad (x\in V),
\]
and $h' \in V(\infty)$ by Theorem \ref{formofh}. As $h_{a_k}\to h$, we know that $h'=h$, and hence $e_I=c_J$ and
\[
\sum_{i\in I}\alpha_i e_i =\sum_{j\in J} \beta_jc_j
\]
by Corollary \ref{h=h'}.

We can relabel the $\alpha_i$'s such that $I=\{1,\ldots,p\}$ and $0=\alpha_1\leq\alpha_2\leq\ldots\leq\alpha_p$. It follows from Remark \ref{uuni} that $p=q$ and $\alpha_i =\beta_i$ for all $i\in \{1,\ldots,p\}$. Moreover, $\sum_{i=1}^p e^{-\alpha_i}e_i  =\sum_{i=1}^p e^{-\beta_i}c_i$.

As
\[
\phi(a_m) = \frac{\sum_{i=1}^r (e^{\mu_{im}} - e^{-\mu_{im}})c_{im}}{\sum_{i=1}^r e^{\mu_{im}} - e^{-\mu_{im}}}
= \frac{\sum_{i=1}^r (e^{-\beta_{im}} - e^{-\mu_{1m}-\mu_{im}})c_{im}}{\sum_{i=1}^r e^{-\beta_{im}} - e^{-\mu_{1m}-\mu_{im}}},
\]
we find  that
\[
\lim_{m} \phi(a_m) = (\sum_{i=1}^p e^{-\beta_i})^{-1}\sum_{i=1}^p e^{ -\beta_i}c_i = (\sum_{i=1}^p  e^{-\alpha_i})^{-1}\sum_{i=1}^p e^{-\alpha_i}e_i.
\]
Thus,  $ \phi(a_m)\to \phi(h)$, which completes the proof.
 \end{proof}

To prove continuity of $\phi$ on the boundary $V(\infty)$, we need the following technical lemma.

\begin{lemma} \label{lem:5.17}
Let $J\subset \{1,\ldots,r\}$ be nonempty, and for each $n$, let $\{c_{jn}\colon j\in J\}$ be a collection of mutually orthogonal minimal  tripotents in $V$ such that  $c_{jn}\to c_j$ for all $j\in J$. For $j\in J$ let  $(\beta_{jn})_n$ a sequence in $[0, \infty)$ converging to $\beta_j\in [0,\infty]$, with $\min_{j\in J}\beta_{jn} =0$ for each $n$.  If we let $J'= \{j\in J\colon \beta_j<\infty\}\neq \emptyset$ and consider the horofunctions,
\[
h_n(x) = \Lambda_{V_2(c_{Jn})}(-\frac{1}{2}(c_{Jn}\bo P_2(c_{Jn})x+P_2(c_{Jn})x\bo c_{Jn})-\sum_{j\in J}\beta_{jn} (c_{jn}\bo c_{jn})) \qquad (x\in V),
\]
 where $c_{Jn} = \sum_{i\in J} c_{jn}$, then
  \begin{equation}\label{eq:5.17.0}
\lim_{n\to\infty} h_n(x) =
 \Lambda_{V_2(c_{J'})}(-\frac{1}{2}(c_{J'}\bo P_2(c_{J'})x+P_2(c_{J'})x\bo c_{J'})-\sum_{j\in J'}\beta_{j} (c_j\bo c_j))
  \end{equation}
with $c_{J'} = \sum_{j\in J'} c_j$.
\end{lemma}
\begin{proof} We show that each subsequence of $(h_n(x))$ has a convergent subsequence with limit the right-hand side of (\ref{eq:5.17.0}).
Pick a subsequence $(h_{k}(x))$. As $\{u\in V\colon \langle u,u\rangle\ =1\}$ is compact, there exists $w^k\in V_2(c_{Jk})$ with
$\langle w^k,w^k\rangle =1$ and
   \[
h_{k}(x)
=
\langle (-\frac{1}{2}(c_{Jk}\bo P_2(c_{Jk})x+P_2(c_{Jk})x\bo c_{Jk})-\sum_{j\in J}\beta_{jk} (c_{jk}\bo c_{jk}))w^k,w^k\rangle.
   \]
    Taking a subsequence, we may assume that $w^k\to w$. For each $k$, let
\begin{equation}\label{w}
V = \bigoplus_{0\leq s\leq t; s,t\in J} V_{st}^k
\end{equation}
be the Peirce decomposition of $V$ with respect to the tripotents $\{c_{jk}\colon j\in J\}$, and let
$$  V = \bigoplus_{0\leq s\leq t; s,t\in J} V_{st}$$
be the decomposition with respect to the tripotents $\{c_j\colon j\in J\}$.  We show

 \begin{equation}\label{eq:5.17.1}
 w\in\bigoplus_{s,t\in J'\colon s\leq t} V_{st} = V_2(c_{J'}).
\end{equation}
Let $w_{st}^k \in V_{st}^k$ be the $(s,t)$-component of $w^k$ in the Peirce decomposition  in (\ref{w}).
 Then we have $w_{st}^k \to w_{st}\in V_{st}$ and
\[
-\frac{1}{2}(c_{Jk}\bo P_2(c_{Jk})x+P_2(c_{Jk})x\bo c_{Jk})w_{st}^k \to -\frac{1}{2}(c_{J}\bo P_2(c_{J})x+P_2(c_{J})x\bo c_{J})w_{st}.
\]
Moreover,
\[
\sum_{j\in J} \beta_{jk} (c_{jk}\bo c_{jk}) w^k_{st} = \sum_{j\in J} \left( \frac{\delta_{sj}\beta_{jk} +\delta_{tj}\beta_{jk}}{2}\right) w^k_{st}.
\]
Recall that if $s,t\in J$ with $\{s,t\}\not\subset J'$, then $\beta_{sk}\to\infty$ or $\beta_{tk}\to\infty$. As $h_k(x)\geq -\|x\|$ for all $k$ and $\beta_{jk}\geq 0$ for all $j\in J$, we find that $w^k_{st}\to 0$ for all $s,t\in J$ with $\{s,t\}\not\subset J'$. This implies that (\ref{eq:5.17.1}) holds.

Next, we  make use of the following fact. Let $e$ be a tripotent. If $v\in V_2(e)\cup V_0(e)$, $y\in V$ and $z\in V_2(e)$, then
$P_2(e)\{v,y,z\} =\{P_2(e)v,P_2(e)y,z\}$ and $P_2(e)\{y,v,z\} =\{P_2(e)y,P_2(e)v,z\}$.
In particular,
\begin{eqnarray*}
\langle -\frac{1}{2}(c_{J}\bo P_2(c_{J})x+P_2(c_{J})x\bo c_{J})w,w\rangle
 &=& \langle -\frac{1}{2}(c_{J'}\bo P_2(c_{J'})P_2(c_{J})x+P_2(c_{J'})P_2(c_{J})x\bo c_{J'})w,w\rangle\\
 & = & \langle -\frac{1}{2}(c_{J'}\bo P_2(c_{J'})x+P_2(c_{J'})x\bo c_{J'})w,w\rangle,
 \end{eqnarray*}
 as $P_2(c_{J'})P_2(c_{J}) = P_2(c_{J'})$.

 It follows that
 \begin{eqnarray*}
 \limsup_{k\to\infty} h_{k}(x) & = &
 \limsup_{k\to\infty} \langle (-\frac{1}{2}(c_{Jk}\bo P_2(c_{Jk})x+P_2(c_{Jk})x\bo c_{Jk})-\sum_{j\in J}\beta_{jk} (c_{jk}\bo c_{jk}))w^k,w^k\rangle\\
  & \leq & \lim_{k\to\infty}
  \langle (-\frac{1}{2}(c_{Jk}\bo P_2(c_{Jk})x+P_2(c_{Jk})x\bo c_{Jk})-\sum_{j\in J'}\beta_{jk} (c_{jk}\bo c_{jk}))w^k,w^k\rangle\\
   & = & \langle (-\frac{1}{2}(c_{J'}\bo P_2(c_{J'})x+P_2(c_{J'})x\bo c_{J'} - \sum_{j\in J'}\beta_{j} (c_j\bo c_j))w,w\rangle.
 \end{eqnarray*}

On the other hand,
\begin{multline*}
\liminf_{k\to\infty} h_{k}(x) \geq
 \liminf_{k\to\infty} \langle (-\frac{1}{2}(c_{Jk}\bo P_2(c_{Jk})x+P_2(c_{Jk})x\bo c_{Jk})\\
 -\sum_{j\in J}\beta_{jk} (c_{jk}\bo c_{jk}))P_2(c_{J'k})w^k,P_2(c_{J'k})w^k\rangle\langle P_2(c_{J'k})w^k,P_2(c_{J'k})w^k\rangle^{-1}\\     \to \langle (-\frac{1}{2}(c_{J'}\bo P_2(c_{J'})x+P_2(c_{J'})x\bo c_{J'} - \sum_{j\in J'}\beta_{j} (c_j\bo c_j))w,w\rangle.
\end{multline*}
Thus,
\[
\lim_{k\to\infty} h_{k}(x) = \langle (-\frac{1}{2}(c_{J'}\bo P_2(c_{J'})x+P_2(c_{J'})x\bo c_{J'} - \sum_{j\in J'}\beta_{j} (c_j\bo c_j))w,w\rangle.
\]
Since $w\in V_2(c_{J'})$, we find  that
\[
\lim_{k\to\infty} h_{k}(x)\leq  \Lambda_{V_2(c_{J'})}(-\frac{1}{2}(c_{J'}\bo P_2(c_{J'})x+P_2(c_{J'})x\bo c_{J'})-\sum_{j\in J'}\beta_{j} (c_j\bo c_j)).
\]

To show that this is an equality, let $u\in V_2(c_{J'})$ be such that
\begin{multline*}
\Lambda_{V_2(c_{J'})}(-\frac{1}{2}(c_{J'}\bo P_2(c_{J'})x+P_2(c_{J'})x\bo c_{J'})-\sum_{j\in J'}\beta_{j} (c_j\bo c_j))\\
 =\langle  (-\frac{1}{2}(c_{J'}\bo P_2(c_{J'})x+P_2(c_{J'})x\bo c_{J'})-\sum_{j\in J'}\beta_{j} (c_j\bo c_j))u,u\rangle.
 \end{multline*}
 As $P_2(c_{J'k})u\in V_2(c_{Jk})$, we have that
 \begin{multline*}
  \langle (-\frac{1}{2}(c_{Jk}\bo P_2(c_{Jk})x+P_2(c_{Jk})x\bo c_{Jk})-\sum_{j\in J}\beta_{jk} (c_{jk}\bo c_{jk}))w^k,w^k\rangle \\
  \geq
   \langle (-\frac{1}{2}(c_{Jk}\bo P_2(c_{Jk})x+P_2(c_{Jk})x\bo c_{Jk})\\
   -\sum_{j\in J}\beta_{jk} (c_{jk}\bo c_{jk}))P_2(c_{J'k})u,P_2(c_{J'k})u\rangle \langle P_2(c_{J'k})u,P_2(c_{J'k})u\rangle^{-1}.
\end{multline*}
Note that, as $\langle P_2(c_{J'k})u,P_2(c_{J'k})u\rangle\to \langle u,u\rangle =1$, the right-hand side is defined for all $k$ large.

So,
 \begin{multline*}
   \langle (-\frac{1}{2}(c_{Jk}\bo P_2(c_{Jk})x+P_2(c_{Jk})x\bo c_{Jk})
   -\sum_{j\in J}\beta_{jk} (c_{jk}\bo c_{jk}))P_2(c_{J'k})u,P_2(c_{J'k})u\rangle\\
    =  \langle (-\frac{1}{2}(c_{J'k}\bo P_2(c_{J'k})x+P_2(c_{J'k})x\bo c_{J'k})
   -\sum_{j\in J'}\beta_{jk} (c_{jk}\bo c_{jk}))P_2(c_{J'k})u,P_2(c_{J'k})u\rangle \\
   \to  \langle (-\frac{1}{2}(c_{J'}\bo P_2(c_{J'})x+P_2(c_{J'})x\bo c_{J'})
   -\sum_{j\in J'}\beta_{j} (c_j\bo c_j))u,u\rangle,
\end{multline*}
which proves (\ref{eq:5.17.0}).
\end{proof}

\begin{lemma}\label{ctinf} The map $\phi\colon V \cup V(\infty) \to {D}^\circ$ is continuous on $V(\infty)$.
\end{lemma}
\begin{proof}
Let $h_n\to h$ in $V(\infty)$.
We show that each subsequence  $(\phi(h_{k}))$ of $(\phi(h_{n}))$ has a convergent subsequence with limit $\phi(h)$.
As $h_n$ is horofunction,  we can express it as
\[
h_n(x) = \Lambda_{V_2(c_{J_nn})}(-\frac{1}{2}(c_{J_nn}\bo P_2(c_{J_nn})x+P_2(c_{J_nn})x\bo c_{J_nn})-\sum_{j\in J_n}\beta_{jn} (c_{jn}\bo c_{jn}))\qquad (x\in V).
\]
After taking a successive subsequences we may assume that $J_{k} =J$ for all $k$,  $c_{jk} \to c_j$ and $\beta_{jk}\to\beta_j\in[0,\infty]$ for all $j\in J$. Note that $\min_{j\in J} \beta_j =0$, as  $\min_{j\in J} \beta_{jk} =0$ for all $k$.

Now let $J'= \{j\in J\colon \beta_j<\infty\}$. By Lemma \ref{lem:5.17}, we find  that $h_k(x) \to h'(x)$ for each $x\in V$, where
\[
h'(x) =  \Lambda_{V_2(c_{J'})}(-\frac{1}{2}(c_{J'}\bo P_2(c_{J'})x+P_2(c_{J'})x\bo c_{J'})-\sum_{j\in J'}\beta_{j} (c_j\bo c_j)).
\]
By Theorem \ref{formofh} we know that $h'$ is a horofunction.  As $h_n\to h$, we conclude that $h=h'$, and hence  $e_I = c_{J'}$ and $\sum_{i\in I}\alpha_i e_i = \sum_{j\in J'} \beta_jc_j$ by Corollary \ref{h=h'}.
This implies, by Remark \ref{uuni}, that
\[
\lim_{k\to\infty} h_{k}(x) = \lim_{k\to\infty}\frac{\sum_{j\in J} e^{-\beta_{jk}} c_{jk}}{\sum_{j\in J} e^{-\beta_{jk}} }=
\frac{\sum_{j\in J'} e^{-\beta_j }c_j}{\sum_{j\in J'} e^{-\beta_j} }= \frac{\sum_{i\in I} e^{-\alpha_i}e_i}{\sum_{i\in I} e^{-\alpha_i}}  = \phi(h)
\]
and hence we are done.
\end{proof}

Collecting the results it now easy to show that $\phi$ is a homeomorphism.
\begin{proof}[Proof of Theorem \ref{homeo}]
Note that $\phi$ is continuous on $V\cup V(\infty)$ by Lemmas \ref{cty} and \ref{ctinf}. Moreover, $\phi$ is a bijection by Lemmas \ref{=}, \ref{sur2}, and \ref{inj}.  As $V\cup V(\infty)$ is compact and $D^\circ$ is Hausdorff, we conclude that $\phi$ is a homeomorphism.
\end{proof}

\section{Geometry of $V\cup V(\infty)$} \label{GeomComp}

We now analyse the geometry of the metric compactification of $V$. Recall that on $V(\infty)$ there is a natural equivalence relation,   $h\sim g$ if $\sup_{x\in V} |h(x)-g(x)|<\infty$.  In this section, we show that the partition of $V(\infty)$ into equivalence classes is closely related to the geometry of $D^\circ$  (and hence also to $D^*$). In fact, we prove that the homeomorphism $\phi\colon V\cup V(\infty)\to D^\circ$ given in (\ref{phi1}) and (\ref{phi2}) maps each equivalence class  onto the relative interior of a boundary face of $D^\circ$.

For the basic terminology from convex analysis we follow \cite{Rock}.  If $C\subseteq V$ is convex, then $F\subseteq C$ is called
a {\em face} if $\lambda x+(1-\lambda)y\in F$ for some $0<\lambda<1$ and $x,y\in C$ implies that $x,y\in F$.  Note that the empty set and $C$ are both faces of $C$, and each face is convex. The {\em relative interior} of a face $F$, denoted $\mathrm{ri}\, F$, is the interior of $F$ regarded as a subset of the affine hull of $F$.   It is well known that each nonempty convex set $C$ is partitioned by the relative interiors of its nonempty faces, see \cite[Theorem 18.2]{Rock}.

To analyse the equivalence classes it is useful to recall  that for Busemann points $h,g\in V(\infty)$ one has that $h\sim g$ if and only if $\delta(h,g)<\infty$, see \cite[Proposition 4.5]{Wa1}. As each horofunction is a Busemann point by Corollary \ref{buse}, we see that the equivalence classes coincide with the parts of $V(\infty)$. Therefore we start by analysing the parts.

Using (\ref{detourcost}) and Remark \ref{rm}  we find for $h,h'\in V(\infty)$ that
\begin{equation}\label{H(h,h')}
H(h,h') = \lim_{t\to\infty} \|\psi(t)\|+h'(\psi(t)),
\end{equation}
where
\begin{equation}\label{formh21}
h(x) =\Lambda_{V_2(e_I)}(-\frac{1}{2}(e_I\bo P_2(e_I)x+P_2(e_I)x\bo e_I)-\sum_{i\in I}\alpha_i (e_i\bo e_i)) \quad (x\in V)
\end{equation}
and $\psi(t) = te_I -\sum_{i\in I}\alpha_i e_i$.
Likewise,  $H(h',h) = \lim_{t\to\infty} \|\psi' (t)\|+h(\psi'(t))$, where
\begin{equation}\label{formh3}
h'(x) =\Lambda_{V_2(c_J)}(-\frac{1}{2}(c_J\bo P_2(c_J)x+P_2(c_J)x\bo c_J)-\sum_{j\in J}\beta_j (c_j\bo c_j)) \quad (x\in V)
\end{equation}
and $\psi'(t) = tc_J -\sum_{j\in J}\beta_j c_j$.

\begin{lemma}\label{eicj1}
Let $h$ and $h'$ be given by (\ref{formh21}) and (\ref{formh3}). If  $e_I=c_J$,
then
\[
H(h,h')  =  \sup_{u\in V_2(e_I)\colon \langle u,u\rangle =1} \langle ((a-b)\bo e_I)u,u\rangle
\]
and
\[
H(h',h) =  \sup_{u\in V_2(e_I)\colon \langle u,u\rangle =1} \langle ((b-a)\bo e_I)u,u\rangle.
\]
where $a=\sum_{i\in I}\alpha_i e_i $ and $b=\sum_{j\in J} \beta_jc_j$.
\end{lemma}
\begin{proof} Let $\psi(t) = te_I -\sum_{i\in I}\alpha_i e_i$.  For $t\geq 0$ large, $\|\psi(t)\|=t$, and hence
\begin{multline*}
\|\psi(t)\|+h'(\psi(t))  =  t + \sup_{u\in V_2(c_J)\colon \langle u,u\rangle =1} \langle -\frac{1}{2}(c_J\bo \psi(t) + \psi(t)\bo c_J)u - \sum_{j\in J}\beta_j (c_j\bo c_j)u,u\rangle\\
 =   \sup_{u\in V_2(c_J)\colon \langle u,u\rangle =1} \langle t(c_J\bo c_J)u-\frac{1}{2}(c_J\bo (te_I -a) + (te_I-a)\bo c_J)u - \sum_{j\in J}\beta_j (c_j\bo c_j)u,u\rangle\\
=   \sup_{u\in V_2(c_J)\colon \langle u,u\rangle =1} \langle \frac{1}{2}(e_I\bo a + a\bo e_I)u - \sum_{j\in J}\beta_j (c_j\bo c_j)u,u\rangle\\
  =   \sup_{u\in V_2(c_J)\colon \langle u,u\rangle =1} \langle (a\bo e_I)u -(b\bo c_J)u,u\rangle \\
 =   \sup_{u\in V_2(e_I)\colon \langle u,u\rangle =1} \langle ((a-b)\bo e_I)u,u\rangle.\\
\end{multline*}
So, by (\ref{H(h,h')})  the first equality holds. The second one is obtained by changing the roles of $h$ and $h'$.
\end{proof}

We use this lemma to give a simple criterion  for two horofunctions to be in the same part of $V(\infty)$.

\begin{theorem}\label{parts2} Two horofunctions $h$ and $h'$ given by (\ref{formh21}) and (\ref{formh3}), respectively, are in the same part
 of $V(\infty)$ if and only if $e_I =c_J$,
\end{theorem}
\begin{proof} Let  $e_I=c_J$. Then by Lemma \ref{eicj1}, we have that
\[
\delta(h,h') = \sup_{u\in V_2(e_I)\colon \langle u,u\rangle =1} \langle ((a-b)\bo e_I)u,u\rangle +\sup_{u\in V_2(e_I)\colon \langle u,u\rangle =1} \langle ((b-a)\bo e_I)u,u\rangle <\infty
\]
where $a=\sum_{i\in I}\alpha_i e_i $ and $b=\sum_{j\in J} \beta_jc_j$. Hence $h$ and $h'$ are in the same part.

Conversely,  given $e_I\neq c_J$, we need to show $h$ and $h'$ are in different parts, that is, $\delta (h,h')=\infty$. We have either $c_J\nleq e_I$ or  $e_I\nleq c_J$. Assume the former. Note that it suffices to show $H(h,h') =\infty$, since the detour cost is nonnegative.

By Lemma \ref{eicj5} and Remark \ref{+op} we have
\begin{equation}\label{>0}
\langle c_J - \frac{1}{2}(c_J\bo P_2(c_J)e_I +P_2(c_J)e_I\bo c_J)(c_J), \, c_J\rangle >0.
\end{equation}

As before, for large $t$, we have $\|\psi(t)\|=t$, so that
\begin{eqnarray*}
&& \|\psi(t)\|+h'(\psi(t))  =  t + \sup_{u\in V_2(c_J)\colon \langle u,u\rangle =1} \langle -\frac{1}{2}(c_J\bo \psi(t) + \psi(t)\bo c_J)u - \sum_{j\in J}\beta_j (c_j\bo c_j)u,u\rangle\\
& = &  \sup_{u\in V_2(c_J)\colon \langle u,u\rangle =1} \langle tu- \frac{t}{2}(c_J\bo P_2(c_J)e_I + P_2(c_J)e_I\bo c_J)u
+ \,\frac{1}{2}(c_J\bo a + a\bo c_J)u - \sum_{j\in J}\beta_j (c_j\bo c_j)u,u\rangle.
\end{eqnarray*}
Hence (\ref{>0}) implies
\[
H(h,h')\geq \lim_{t\to \infty} \langle tc_J-\frac{t}{2}(c_J\bo P_2(c_J)e_I + P_2(c_J)e_I\bo c_J)(c_J)
+\frac{1}{2}(c_J\bo a + a\bo c_J)c_J - \sum_{j\in J}\beta_j c_j, \,c_J\rangle =\infty.
\]
Analogously,  $e_I\nleq c_J$ implies $H(h',h)=\infty$.
\end{proof}

Let us now recall the facial structure of ${D}^*$. By \cite[Theorem 4.4]{ER}, the closed boundary faces of $D^*$ are exactly the sets of the form
$F^*_e=\{\widetilde x\in {D}^*\colon  \widetilde x(e) = [e, x] =1\}$,
where $e$ is a tripotent in $V$.  So, the boundary faces of $D^\circ$ are precisely the sets of the form
\[
F_e=\{x\in {D}^\circ\colon  [e, x] =1\} \qquad (e\in V\mbox{ tripotent}).
\]
Note that $F_e\subset \partial D^\circ$, as $F_e^*\subset \partial D^*$.

The next lemma gives an alternative description of $F_e$, which will be useful in our discussion.
\begin{lemma}\label{facelem} Suppose that $e\in V$ is a tripotent. Then
 $$F_e =\{ \mbox{$\sum_{i=1}^p \lambda_i e_i$}\colon  \mbox{ $\sum_{i=1}^p\lambda_i =1$, $\lambda_i >0$,
  and $e_i$'s mutually orthogonal minimal tripotents with $e_i\leq e$}\}.$$
\end{lemma}
\begin{proof}
Suppose that $x\in F_e$. Using the spectral decomposition we can write $x$ as $x=  \sum_{i=1}^p \lambda_i e_i$, where $\lambda_p>0$ and $p\leq r$ (so we ignore the zero eigenvalues).  Then $\sum_{i=1}^p \lambda_i =1$, since $x\in\partial D^\circ$.
As $F_e$ is a face and $e_i\in D^\circ$, we know that $e_i\in F_e$, and hence $1=[e,e_i]=[P_2(e_i)e,e_i]$.  Combining this with the fact that  $P_2(e_i)e\in\mathbb{C}e_i$ and $[e_i,e_i]=1$ gives $P_2(e_i)e=e_i$. Thus, $e_i\leq e$ for $i=1,\ldots,p$.

On the other hand, given $x = \sum_{i=1}^p \lambda_i e_i$  such that $\sum_{i=1}^p \lambda_i =1$, $\lambda_i >0$,
and $e_1, \ldots, e_p\leq e$  mutually orthogonal minimal tripotents in $V$, we have that  $x\in D^\circ$ and
\[
[e, x] = \sum_{i=1}^p \lambda_i [ e, e_i]
=  \sum_{i=1}^p \lambda_i [ e_i, e_i ] =1.
\]
Hence $x\in F_e$.
\end{proof}
We like to point out that $F_e = \partial D^\circ \cap A(e)_+$, where $A(e)_+$ is the closed positive cone in the $\mathrm{JB}$-algebra $A(e)$ in $V_2(e)$, cf.\,\cite[Theorem 6.12]{loos}.

\begin{theorem} \label{geometry}
If $h\in V(\infty)$ is given by  (\ref{formh21}), then $\phi(h)\in \mathrm{ri}\,F_{e_I}$. Moreover, $\phi$ maps each equivalence class in $V(\infty)/\sim$  onto the relative  interior of a boundary face of $D^\circ$.
\end{theorem}
\begin{proof}  Let $q=|I|>0$ and $w = q^{-1}e_I$, so $[e_I,w]=1$ and $w\in F_{e_I}$.  We claim that $w$ is in the relative interior of $F_{e_I}$. Let $x\in F_{e_I}$. To prove the claim it suffices to show that for each $\epsilon >0$ small, $w_\epsilon = w +\epsilon (w-x)$ is in $F_{e_I}$, see \cite[Theorem 6.4]{Rock}.

By Lemma \ref{facelem} we know  that we can  write $x = \sum_{i=1}^p \lambda_i e_i$, where $\sum_{i=1}^p \lambda_i =1$, $\lambda_i>0$ and $e_i\leq e_I$ for all $i$.  We have that $e_I-\sum_{i=1}^p e_i=\sum_{i=p+1}^{q}e_i$ is a (possibly $0$)  tripotent.  So,
\begin{equation}\label{sum}
w_\epsilon = \sum_{i=1}^p (q^{-1}(1+\epsilon) -\epsilon\lambda_i)e_i+ \sum_{i=p+1}^{q} q^{-1}(1+\epsilon)e_i
\end{equation}
and $[e_I,w_\epsilon] = (1+\epsilon)[e_I,w] - \epsilon[e_I,x] = 1$. As $q^{-1}(1+\epsilon) -\epsilon\lambda_i>0$ for all $\epsilon >0$ small, the right hand-side of (\ref{sum}) is a spectral decomposition of $w_\epsilon$ for all $\epsilon>0$ small. (We have ignored terms with zero eigenvalues.) Thus, $w_\epsilon \in  D^\circ$ for all $\epsilon >0$ small, and hence $w_\epsilon \in F_{e_I}$.

To complete the proof of the first assertion we assume by way of contradiction  that $\phi(h)\not\in \mathrm{ri}\, F_{e_I}$. As $\phi(h)\in F_{e_I}$, we know that $\phi(h)$ is in the relative boundary of $F_{e_I}$. This implies that $z_\epsilon = (1+\epsilon)\phi(h) -\epsilon w\not\in F_{e_I}$ for all $\epsilon >0$. Here we use the fact that $w$ is in $\mathrm{ri}\, F_{e_I}$ and $F_{e_I}$ is a convex set.

Note that
\begin{equation}\label{sumz}
z_\epsilon = \sum_{i=1}^q \left(\frac{(1+\epsilon)e^{-\alpha_i}}{\sum_{j=1}^q e^{-\alpha_j}} -\frac{\epsilon}{q}\right)e_i.
\end{equation}
Let $\mu^\epsilon_i$ be the coefficient of the $e_i$ in the sum (\ref{sumz}). Then $\mu^\epsilon_i>0$ for all $i$ when   $\epsilon >0$ is sufficiently small. Moreover, $\sum_{i=1}^q\mu^\epsilon_i=[e_I,z_\epsilon]  = 1$, since $e_I=e_1+\cdots+e_q$. But this implies that $z_\epsilon \in F_{e_I}$ for all $\epsilon >0$ small, which is impossible.

To show the second statement we note that for $h,h'\in V(\infty)$ given by (\ref{formh21}) and (\ref{formh3}), respectively, we have that $h\sim h'$ if and only if they are in the same part, as all horofunctions are Busemann points. This is equivalent to saying that $e_I=c_J$ by Theorem \ref{parts2}. So by the first assertion, we get that $h$ and $h'$ are both  mapped into $\mathrm{ri}\,F_{e_I}$ under $\phi$. As $\phi$ maps $V(\infty)$ onto $\partial D^\circ$ we conclude that $\phi$ maps each equivalence class in $V(\infty)/\sim$ onto the relative interior of a boundary face of $D^\circ$. \end{proof}

In the remained of this section we show that each part in $V(\infty)$ with the detour distance  is isometric to a normed space.  To introduce the norm let $e\in V$ be a tripotent and define for $x\in A(e)$,
\begin{eqnarray}\label{var}
\|x\|_{\var} &=& \sup_{u\in V_2(e)\colon \langle u,u\rangle=1} \langle (x\bo e)u,u\rangle +\sup_{u\in V_2(e)\colon \langle u,u\rangle=1} \langle (-x\bo e)u,u\rangle  \notag \\
& = &\sup_{u\in V_2(e)\colon \langle u,u\rangle=1} \langle (x\bo e)u,u\rangle -\inf_{u\in V_2(e)\colon \langle u,u\rangle=1} \langle (x\bo e)u,u\rangle.
\end{eqnarray}
\begin{lemma} \label{varnorm} The function $\|\cdot\|_{\var}$ is a semi-norm on the real vector space $A(e)$, with $\|x\|_{\var}=0$ if and only if $x=\lambda e$ for some $\lambda \in \R$.
\end{lemma}
\begin{proof} From  the definition, we have $\|x\|_{\var}\geq 0$ for all $x\in A(e)$. If $\alpha\geq 0$, then $\|\alpha x\|_{\var} =\alpha\|x\|_{\var}$ by (\ref{var}). For $\alpha<0$ we have
\begin{eqnarray*}
\|\alpha x\|_{\var} &=& \sup_{u\in V_2(e)\colon \langle u,u\rangle=1} \langle (\alpha x\bo e)u,u\rangle -\inf_{u\in V_2(e)\colon \langle u,u\rangle=1} \langle (\alpha x\bo e)u,u\rangle  \notag \\
& = &-\inf_{u\in V_2(e)\colon \langle u,u\rangle=1} \langle (-\alpha x\bo e)u,u\rangle +\sup_{u\in V_2(e)\colon \langle u,u\rangle=1} \langle (-\alpha x\bo e)u,u\rangle\\
& = & -\alpha\|x\|_{\var},
 \end{eqnarray*}
 and hence $\|\alpha x\|_\var = |\alpha|\|x\|_\var$ for all $\alpha\in \R$ and $x\in A(e)$.

It follows directly from the definition  that $\|x+y\|_\var\leq \|x\|_\var+\|y\|_\var$ for all $x,y\in A(e)$. Given $\lambda\in\R$ and $x\in A(e)$, we have
\begin{eqnarray*}
\|x+\lambda e\|_{\var} &=& \sup_{u\in V_2(e)\colon \langle u,u\rangle=1} \langle ((x+\lambda e)\bo e)u,u\rangle -\inf_{u\in V_2(e)\colon \langle u,u\rangle=1} \langle ((x+\lambda e)\bo e)u,u\rangle  \\
&=& \sup_{u\in V_2(e)\colon \langle u,u\rangle=1} (\langle (x\bo e)u,u\rangle +  \lambda\langle (e\bo e)u,u\rangle)  -\inf_{u\in V_2(e)\colon \langle u,u\rangle=1} (\langle (x\bo e)u,u\rangle +  \lambda\langle (e\bo e)u,u\rangle) \\
&=& \sup_{u\in V_2(e)\colon \langle u,u\rangle=1} (\langle (x\bo e)u,u\rangle +  \lambda)  -\inf_{u\in V_2(e)\colon \langle u,u\rangle=1} (\langle (x\bo e)u,u\rangle +  \lambda) \\
& = &  \sup_{u\in V_2(e)\colon \langle u,u\rangle=1} \langle (x\bo e)u,u\rangle -\inf_{u\in V_2(e)\colon \langle u,u\rangle=1} \langle (x\bo e)u,u\rangle,
 \end{eqnarray*}
and hence $\|x+\lambda e\|_{\var}= \|x\|_{\var}$.

On the other hand, if $\|x\|_{\var}=0$ with a spectral decomposition $x = \sum_{i=1}^p  \alpha_ie_i$ in $V(e)$,
where $\alpha_1\geq\ldots\geq\alpha_p\geq 0$ and $e_1+\cdots+e_p=e$, then we show  $x = \lambda e$ for some $\lambda\in \R$.

Take $\mu  = \inf_{u\in V_2(e)\colon \langle u,u\rangle=1}\langle (x\bo e)u,u\rangle$ and set $y = x-\mu e$. So, $\|y\|_\var=\|x\|_\var =0$ and
\[
\|y\|_\var =  \sup_{u\in V_2(e)\colon \langle u,u\rangle=1} \langle (y\bo e)u,u\rangle -\inf_{u\in V_2(e)\colon \langle u,u\rangle=1} \langle ( y\bo e)u,u\rangle =  \sup_{u\in V_2(e)\colon \langle u,u\rangle=1} \langle ( y\bo e)u,u\rangle.
\]
But $y = \sum_{i=1}^p (\alpha_i -\mu)e_i$, and hence for each $e_k\in V_2(e)$ with $ 1\leq k\leq p$ we get that
 \[
\langle ( y\bo e)e_k,e_k\rangle = (\alpha_k-\mu) \langle ( e_k\bo e_k)e_k,e_k\rangle = (\alpha_k-\mu)\langle e_k,e_k\rangle\leq 0.
 \]
 This implies that $\alpha_k-\mu\leq 0$ for all $1\leq k\leq p$ and hence $-y$ is in the closed cone of  the $\JB$-algebra $A(e)$, so that $x\leq \mu e$ in $A(e)$. Likewise
 \[
 \sup_{u\in V_2(e)\colon \langle u,u\rangle=1} \langle (-y\bo e)u,u\rangle = -\inf_{u\in V_2(e)\colon \langle u,u\rangle=1} \langle (y\bo e)u,u\rangle =0,
 \]
 gives $y\geq 0$ and hence $\mu e\leq x$. We conclude that $x = \mu e$, which completes the proof.
 \end{proof}
The preceding result shows that $\|\cdot\|_\var$ is genuine norm on the quotient space $A(e)/\R e$ of the $\JB$-algebra $A(e)$.
Further, we have the following corollary.

\begin{corollary}\label{deltaVinfty}
For $h\in V(\infty)$ we have that $([h],\delta)$ is isometric to $(A(e_I)/\R e_I,\|\cdot\|_\var)$.
\end{corollary}
\begin{proof}
Let $h\in V(\infty)$ be given by (\ref{formh21}). We define a map $\tau\colon [h]\to A(e_I)/\R e_I$ by
$$\tau(h') =  \sum_{j\in J} \beta_jc_j + \R e_I  \in A(e_I)/\R e_I \qquad (h'\in [h] \mbox{ given by  (\ref{formh3})}).$$
This is a bijection, as $\min_{j\in J} \beta_j =0$ and $e_I=c_J$ for all horofunctions $h'\in[h]$ by Theorem \ref{parts2}. It is an isometry by Lemmas \ref{eicj1} and \ref{varnorm}.
\end{proof}

\section{Extension of the exponential map} \label{ExpHom}
The exponential map  $\Exp\colon V\to D$ of the Bergman metric at $0\in D$   is a real analytic homeomorphism, where
\begin{equation}\label{ex1}
\Exp(x) = \tanh(x) = \sum_{i=1}^r \tanh(\lambda_i)e_i
\end{equation}
for each $x\in V$ with spectral decomposition $x= \sum_{i=1}^r\lambda_i e_i$ by  \cite[Lemma 4.3 and Corollary 4.8]{loos}.

In this final section we show that $\Exp$ extends as a homeomorphism
$\widetilde{\Exp} \colon V \cup V(\infty) \to D \cup D(\infty)$ such that $\widetilde{\Exp}$ maps each equivalence class in $V(\infty)/\sim$ onto an equivalence class of $D(\infty)/\sim$. In particular, we find that the metric compactification of a Hermitian symmetric space $M\approx D\subset V$
can be realised as the closed dual unit ball $D^*$, by Theorem \ref{homeo}, and its geometry coincides with the facial structure of $D^*$ by Theorem \ref{geometry}.

Given $h\in V(\infty)$ with
\[
h(x) =  \Lambda_{V_2(e)}(-\frac{1}{2}(e\bo P_2(e)x+P_2(e)x\bo e)-\sum_{i=1}^p\alpha_i (e_i\bo e_i))\quad (x\in V),
 \]
 we define  $\widetilde{\Exp} (h) =g$, where the function $g\colon D \to \mathbb{R}$  is given by
 \begin{equation}\label{ex2}
 g(z) =\frac{1}{2} \log \left\|\sum_{1\leq i\leq j\leq p}e^{-\alpha_i}e^{-\alpha_j}B(z,z)^{-1/2}B(z,e)P_{ij}\right\|
\qquad (z\in D),
 \end{equation}
which is a horofunction by Theorem \ref{hbsd}, as  $\min_i \alpha_i =0$ implies $\max_i e^{-\alpha_i}= 1$.

We will prove that the extension $\widetilde\Exp$ is a homeomorphism in following theorem.

\begin{theorem}\label{exp}
The extension $\widetilde{\Exp} \colon V \cup V(\infty) \to D \cup D(\infty)$ of the exponential map
is a well-defined homeomorphism that maps each equivalence class of $V(\infty)/\sim$ onto an equivalence class of $D(\infty)/\sim$.
\end{theorem}

It follows from this theorem that the geometry and global topology of the metric compactifications of $(D,\rho)$ and of the $\mathrm{JB}^*$-triple $(V,\|\cdot\|)$ with open unit ball $D$ coincide. So have the following consequence by Theorems \ref{homeo} and \ref{geometry}.

\begin{corollary} There exists a homeomorphism $\psi\colon D\cup D(\infty)\to D^\circ$ that maps each equivalence class in $D(\infty)/\sim$ onto the relative interior of a boundary face of $D^\circ$.
\end{corollary}

To prove Theorem \ref{exp}, we will  need the fact that all horofunctions of  $D \cup D(\infty)$ are Busemann points
and exploit the detour distance on the parts in $D(\infty)$.

Let $h$ be horofunction functions in $D(\infty)$  of the form,
\begin{equation}\label{horh}
h(z) =\frac{1}{2} \log \left\|\sum_{1\leq i\leq j\leq p}\lambda_i\lambda_jB(z,z)^{-1/2}B(z,e)P_{ij}\right\|
\qquad (z\in D)
\end{equation}
for some $p\in \{1, \ldots, r\}$,  $\lambda_i\in (0,1]$ with $\max_i \lambda_i =1$ and $e = e_1 + e_2 + \cdots +  e_{p}$ a tripotent.  Using
$V_2(e)= \bigoplus_{1\leq i\leq j\leq p} V_{kl}$ and (\ref{Peirce-2 of sum}),  the norm in (\ref{horh}) can be computed over $V_2(e)$, that is,
\begin{equation}\label{penorm}
\|\sum_{1\leq i\leq j\leq p}\lambda_i\lambda_jB(z,z)^{-1/2}B(z,e)P_{ij}\|
=\|\sum_{1\leq i\leq j\leq p}\lambda_i\lambda_jB(z,z)^{-1/2}B(z,e)P_{ij}P_2(e)\|.
\end{equation}

For $ i =1,\ldots,p$, let $\alpha_i = -\log \lambda_i\geq 0$, so $\min_i \alpha_i =0$. Let $\gamma \colon [0, \infty) \to D$ be the path
\begin{equation}\label{geodesic}
\gamma(t) = \Exp( te -\sum_i \alpha_i e_i) = \tanh(te -\sum_i \alpha_i e_i) = \sum_i \tanh(t-\alpha_i)e_i \qquad (t\geq 0).
\end{equation}
Note that $\gamma(t)\to e \in \partial D$ as $t\to\infty$, in other words, $\gamma(t)$ goes to infinity in the metric space $(D, \rho)$.
We show below that $\gamma$ is a {\em geodesic} in the metric space $(D, \rho)$, in the sense that
$d(\gamma(s),\gamma(t)) = |s-t|$ for all $s,t\in [0,\infty)$. We will see that $h$ in (\ref{horh}) is a horofunction obtained by taking a sequence $(\gamma(t_k))$ along $\gamma$. For simplicity we say that $\gamma(t)$ {\em converges to $h$},
viewing $h\in D(\infty)$ as an ideal boundary point of $(D, \rho)$.
\begin{lemma}\label{hisBus}
The path $\gamma$ in (\ref{geodesic}) is a geodesic in $(D,\rho)$ converging to the horofunction $h$ in (\ref{horh}),
and $h$ is a Busemann point.
\end{lemma}
\begin{proof} A direct computation gives, for $t\geq s$, that
\[
\|g_{-\gamma(s)}(\gamma(t))\| = \|\sum_{i=1}^p \tanh(t-s)e_i \|= \tanh(t-s)
\]
and hence $\rho(\gamma(t),\gamma(s))=t-s$, which shows that $\gamma$ is a geodesic.

Observe that for sufficiently large $t\geq 0$,  we have  $\|\gamma(t)\| =\tanh(t)$, as $\min_i\alpha_i=0$.
It follows that $1-\|\gamma(t)\|^2 =1-\tanh^2(t)$ for large $t$.

Set $\beta_{it} = \tanh(t-\alpha_i)$ for $ i=1,\ldots,p$,  and put $\beta_{it}=0$ for $i=0$. Then
\[
B(\gamma(t),\gamma(t))^{-1/2} = \sum_{0\leq i\leq j\leq p} (1-\beta^2_{it})^{-1/2}(1-\beta^2_{jt})^{-1/2}P_{ij}.
\]
Using the identity $e^{2x} = (1+\tanh(x))/(1-\tanh(x))$, we get
\[
\left(\frac{1-\tanh^2(t)}{1-\beta_{it}^2}\right)^{1/2} = \frac{e^{-t}(1+\tanh(t))}
{e^{-t+\alpha_i}(1+\tanh(t-\alpha_i))}\to e^{-\alpha_i} \quad {\rm as}~ t\to\infty
\]
for $ i=1,\ldots ,p$. For $i=0$, we have  $(1-\tanh^2(t))/(1-\beta_{it}^2) = 1-\tanh^2(t)\to 0$.

Recall that by Lemma \ref{hhh} and equations (\ref{bid}) and (\ref{gg}), we have for each $z\in D$ that
\begin{eqnarray*}
\lim_{t}h_{\gamma(t)}(z)  &= &  \lim_{t}\frac{1}{2} \log \left\|(1-\tanh^2(t))B(z,z)^{-1/2}B(z,\gamma(t))B(\gamma(t),\gamma(t))^{-1/2}\right\|\\
 & = &  \lim_{t} \frac{1}{2} \log \left\|\sum_{0\leq i\leq j\leq p} \left(\frac{1-\tanh^2(t)}{1-\beta_{it}^2}\right)^{1/2}  \left(\frac{1-\tanh^2(t)}{1-\beta_{jt}^2}\right)^{1/2} B(z,z)^{-1/2}B(z,\gamma(t))P_{ij}\right\|.
\end{eqnarray*}
This implies that
\[
\lim_{t} h_{\gamma(t)}(z) =  \frac{1}{2} \log \left\|\sum_{1\leq i\leq j\leq p}e^{-\alpha_i}e^{-\alpha_j}B(z,z)^{-1/2}B(z,e)P_{ij}\right\|
 =h(z) \qquad (z\in D)
\]
and  shows that $h$ is a Busemann point.
\end{proof}

Let us now analyse the parts of horofunction boundary $D(\infty)$.
\begin{proposition}\label{partsbsd2}
Let $h,h'\in D(\infty)$ with $h$ given as in (\ref{horh}) and $h'$ given by
\begin{equation}\label{horh'}
h'(z) =\frac{1}{2} \log \left\|\sum_{1\leq i\leq j\leq q}\mu_i\mu_jB(z,z)^{-1/2}B(z,c)P'_{ij}\right\|,
\end{equation} where $c =c_1+\cdots+c_q$.
If $h$ and $h'$ are in the same part, then $e=c$.
\end{proposition}
\begin{proof}
If $h$ and $h'$ are in the same part, then $H(h,h')<\infty$.  Let $\gamma(t)$ be the geodesic converging to $h$ given in (\ref{geodesic}). As $h$ and $h'$ are Busemann,
$H(h,h') = \lim_t \rho(0,\gamma(t)) +h'(\gamma(t)) =  \lim_t t +h'(\gamma(t))$ by  (\ref{detourcost}), so that
\[
H(h,h') =
 \lim_t t+\frac{1}{2}\log \| \sum_{1\leq i\leq j\leq q}\mu_i\mu_jB(\gamma(t),\gamma(t))^{-1/2} B(\gamma(t),c)P'_{ij}\|.
\]

If we let $v =\sum_{i=1}^q \mu_i^{1/2}c_i$, then
\[
 \sum_{1\leq i\leq j\leq q}\mu_i\mu_jB(\gamma(t),\gamma(t))^{-1/2} B(\gamma(t),c)P'_{ij} =
B(\gamma(t),\gamma(t))^{-1/2} B(\gamma(t),c)Q_v^2
\]
by \cite[Corollary 3.15]{loos}. So, $H(h,h')<\infty$ implies that
\begin{equation}\label{conv0}
\|B(\gamma(t),\gamma(t))^{-1/2} B(\gamma(t),c)Q_v^2\|\to 0.
\end{equation}

We claim that this implies that $B(e,c)w=0$ for all $w\in V_2(c)$.  To show this we set $\beta_{it} = \tanh(t-\alpha_i)$ for $i=1,
\ldots,q$, and set $\beta_{it}=0$ for $i=0$. Then
\[
B(\gamma(t),\gamma(t))^{-1} = \sum_{0\leq i\leq j\leq q} (1-\beta^2_{it})^{-1}(1-\beta^2_{jt})^{-1}P_{ij}.
\]
We note  that $B(\gamma(t),\gamma(t))$ is a self-adjoint invertible operator on the Hilbert space
$(V, \langle\cdot, \cdot\rangle)$ in (\ref{<>}), by \cite[Lemma 1.2.22]{book1}.
Hence the preceding equation implies
\[
\langle B(\gamma(t),\gamma(t))^{-1} v,v\rangle\geq \langle v,v\rangle \qquad (v\in V).
\]

Suppose that there exists a $w\in V_2(c)$ with $z=B(e,c)w\neq 0$. Now letting $z_t = B(\gamma(t),c)w$, we find that
\begin{equation}\label{posinner}
\langle B(\gamma(t),\gamma(t))^{-1/2}z_t,B(\gamma(t),\gamma(t))^{-1/2}z_t\rangle = \langle B(\gamma(t),\gamma(t))^{-1}z_t,z_t\rangle \geq
\langle z_t,z_t\rangle\to \langle z,z\rangle>0.
\end{equation}

From \cite[Corollary 3.15]{loos} we know that $Q^2_v$ is invertible on $V_2(c)$, with inverse $Q^2_{v^{-1}}$ and $v^{-1}= \sum_{i=1}^q \mu_i^{-1/2}c_i$.  So, if we let $u= Q^2_{v^{-1}}w$, then there exists $\delta>0$ such that for all  large $t>0$,
\[
\|B(\gamma(t),\gamma(t))^{-1/2} B(\gamma(t),c)Q_v^2\|\geq \|B(\gamma(t),\gamma(t))^{-1/2} B(\gamma(t),c)Q_v^2u\|\|u\|^{-1} = \|B(\gamma(t),\gamma(t))^{-1/2} z_t\|\|u\|^{-1}\geq\delta
\]
by (\ref{posinner}), which contradicts (\ref{conv0}).

It now follows from Lemma \ref{eicj5} that $c\leq e$. As $H(h',h) < \infty$ as well, we can interchange the roles of $h$ and $h'$ and
deduce $e\leq c$, hence concluding the proof of  $e=c$.
\end{proof}

In the next proposition we show that $e=c$ is also a sufficient condition for $h$ and $h'$ to be in the same part.
\begin{proposition}\label{deltabsd} Let $h,h'\in D(\infty)$ be given by  (\ref{horh}) and (\ref{horh'}), respectively.
If $e=c$, then  $h$ and $h'$ are in the same part.
Moreover,  $h=h'$ if and only if  $e=c$ and
$ \sum_{i=1}^p \lambda_ie_i=\sum_{i=1}^q \mu_ic_i$.
\end{proposition}
\begin{proof}
Let $e=c$. Then $p=q$.
Since $h$ and $h'$ are Busemann points, we have by (\ref{detourcost}) that
\[
H(h,h') =  \lim_t t +h'(\gamma(t))=
 \lim_t \frac{1}{2}\log \| \exp(2t)\sum_{1\leq i\leq j\leq p}\mu_i\mu_jB(\gamma(t),\gamma(t))^{-1/2} B(\gamma(t),e)P'_{ij}\|,
\]
where  $\gamma(t)$ is the geodesic converging to $h$  in (\ref{geodesic}),
and we have used $e=c$.

Let $ V =\oplus_{0\leq k\leq l\leq p} V_{kl}$ be the Peirce decomposition with respect to $e_1,\ldots,e_p$.
 For  $w_{kl}\in V_{kl}$, we have
\[
B(\gamma(t),\gamma(t))^{-1/2} B(\gamma(t),e)w_{kl} =  \left(\frac{(1-\beta_{kt})(1-\beta_{lt})}{(1-\beta^2_{kt})^{1/2}(1-\beta^2_{lt})^{1/2}}\right)w_{kl}= \exp(-(t-\alpha_k))\exp(-(t-\alpha_l))w_{kl},
\]
where $\beta_{it} = \tanh(t-\alpha_i)$ for $1\leq i\leq p$ and $\beta_{0t}=0$. So,
\[
\lim_t \exp(2t)B(\gamma(t),\gamma(t))^{-1/2} B(\gamma(t),e)w_{kl} = \exp{(\alpha_k)}\exp{(\alpha_l)}w_{kl}= \lambda_k^{-1}\lambda_l^{-1}w_{kl},
\]
where $\alpha_i  = -\log \lambda_i$ for $i=1,\ldots,p$ and $\lambda_0 =1$.

For $w\in V_2(e)=\oplus_{1\leq k\leq l\leq p} V_{kl}$ we have that
\begin{multline*}
\lim_t\exp(2t)B(\gamma(t),\gamma(t))^{-1/2} B(\gamma(t),e)w
=\lim_{t}\exp(2t)\sum_{1\leq k\leq l\leq p}B(\gamma(t),\gamma(t))^{-1/2} B(\gamma(t),e)w_{kl}   \\
 = \sum_{1\leq k\leq l\leq p} \lambda_k^{-1}\lambda_l^{-1}w_{kl}
 =  \sum_{1\leq k\leq l\leq p}\lambda_k^{-1}\lambda_l^{-1}Q^2_ew_{kl}
 =  Q_{a^{-1}}Q_ew,
\end{multline*}
where $a^{-1}= \sum_{k=1}^p \lambda_k^{-1}e_k$ by \cite[Corollary 3.15]{loos}. Recall that by  (\ref{jb*}), $V_2(e)$ carries the structure of a $\mathrm{JB}^*$-algebra, in which
the self-adjoint part $A(e)$  is partially ordered by the cone $A(e)_+ =\{x^2\colon x\in A(e)\}$.
The tripotents $e_1, \ldots, e_p, c_1, \ldots, c_q$ reside in $A(e)_+$ and are idempotents in the $\mathrm{JB}$-algebra $A(e)$. Let $a=\sum_{k=1}^p \lambda_ke_k $ and $b=\sum_{i=1}^q \mu_ic_i$. Then $a$ and $b$ are invertible elements in $A(e)_+$, with inverses $a^{-1}= \sum_{k=1}^p \lambda_k^{-1}e_k$ and $b^{-1} =\sum_{i=1}^q \mu_i^{-1}c_i$, respectively.

So, as $e=c$, we now find for $v\in V_2(e) = V_2(c)$ that
\[
\lim_t\exp(2t)\sum_{1\leq i\leq j\leq q} \mu_i\mu_jB(\gamma(t),\gamma(t))^{-1/2} B(\gamma(t),e)P_{ij}'v  =
\sum_{1\leq i\leq j\leq q} \mu_i\mu_jQ_{a^{-1}}Q_eP_{ij}'v
  =   Q_{a^{-1}}Q_eQ_bQ_ev.
\]
Hence, by (\ref{penorm}),
\begin{equation}\label{penorm1}
H(h,h')= \frac{1}{2}\log\|Q_{a^{-1}}Q_eQ_bQ_eP_2(e)\| = \frac{1}{2}\log\|Q_{a^{-1}}Q_eQ_bQ_e\|  <\infty.
\end{equation}

Interchanging the roles of $h$ and $h'$, we conclude
\[\delta(h,h') = H(h,h')+H(h',h) <\infty.
\]

Finally, given $e=c$ and $a=b$, we can use the identity $Q_eQ_bQ_e = Q_{Q_eb} = Q_b$ to get
$$ Q_{a^{-1}}Q_eQ_bQ_e = \{b^{-1}, \{b, \cdot,b\}, b^{-1}\} = P_2(e),$$
which is the identity operator on $V_2(e)$, and therefore $H(h,h')=0$ by (\ref{penorm1}).
 Likewise, $H(h',h)=0$, so that $\delta(h,h')=0$ and hence $h=h'$.

Conversely, if $h=h'$, then they are in the same part.  By Proposition \ref{partsbsd2}, we have $e=c$, and  (\ref{penorm1})
implies $\|Q_{a^{-1}}Q_eQ_bQ_e\| =1$. In particular,
$$\|\{a^{-1}, b^2, a^{-1}\}\| = \|Q_{a^{-1}}Q_eQ_bQ_e e\| \leq 1\mbox{\quad and\quad }
\|\{b^{-1}, a^2, b^{-1}\}\| = \|Q_{b^{-1}}Q_eQ_aQ_ee\| \leq 1.$$
In  $A(e)$, the first inequality implies $\{a^{-1},b^2, a^{-1}\} \leq e$, by (\ref{a<e}), and hence
$$b^2 = \{a,\{a^{-1},b^2, a^{-1}\}, a\} \leq \{a,e,a\}=  a^2,$$
 whereas the second implies $a^2\leq b^2$. It follows that $a^2=b^2$ and $a=b$, since $a,b \in A(e)_+$.
\end{proof}
We now begin the proof of  Theorem \ref{exp}, which will be split into several lemmas.
Let $$\widetilde\Exp\colon V\cup V(\infty) \longrightarrow D\cup D(\infty)$$ be as defined in (\ref{ex1}) and (\ref{ex2}).
\begin{lemma}\label{exp1}
The map $\widetilde\Exp$ is a well-defined bijection which maps $V$ onto $D$ and $V(\infty)$ onto $D(\infty)$.
Further, it maps each equivalence class  in $V(\infty)/\sim$ onto an equivalence class of $D(\infty)/\sim$.
\end{lemma}
\begin{proof}
To see that $\widetilde\Exp$ is well-defined, pick $h \in V(\infty)$  with two formulations
\[
h(x) =  \Lambda_{V_2(e)}(-\frac{1}{2}(e\bo P_2(e)x+P_2(e)x\bo e)-\sum_{i=1}^p\alpha_i (e_i\bo e_i))\qquad (x\in V)
 \]
and
\[
h(x) =  \Lambda_{V_2(e')}(-\frac{1}{2}(e'\bo P_2(e')x+P_2(e')x\bo e')-\sum_{i=1}^q\gamma_i (e'_i\bo e'_i))\quad (x\in V).
 \]

 Note that
 \[
 \frac{1}{2} \log \left\|\sum_{1\leq i\leq j\leq p}e^{-\alpha_i}e^{-\alpha_j}B(z,z)^{-1/2}B(z,e)P_{ij}\right\| = \frac{1}{2} \log \left\|B(z,z)^{-1/2}B(z,e)Q(a)Q(e)P_2(e)\right\|
\qquad (z\in D),
 \]
where $a = \sum_{i=1}^p \alpha_ie_i$. Likewise,
 \[
 \frac{1}{2} \log \left\|\sum_{1\leq i\leq j\leq q}e^{-\gamma_i}e^{-\gamma_j}B(z,z)^{-1/2}B(z,e')P'_{ij}\right\| = \frac{1}{2} \log \left\|B(z,z)^{-1/2}B(z,e')Q(b)Q(e')P_2(e')\right\|
\qquad (z\in D),
 \]
where $b = \sum_{i=1}^q \gamma_ie'_i$.

It follows from Corollary \ref{h=h'} that $e =e'$, $p=q$ and $a =b$.
By relabelling we may as well assume that $\alpha_1\geq \ldots\geq \alpha_p=0$ and $\gamma_1\geq \ldots\geq \gamma_p=0$. As the eigenvalues in the spectral decomposition in $V$ are unique, we conclude that $\alpha_i=\gamma_i$ for all $i$, and hence $\widetilde\Exp$  is well-defined.

Note that it follows from Theorems \ref{hbsd} and \ref{formofh} that $\widetilde\Exp$ maps $V(\infty)$ onto $D(\infty)$.  Moreover,
given $h'$ is in the same part as $h$, with
\[
h'(x) =  \Lambda_{V_2(c)}(-\frac{1}{2}(c\bo P_2(c)x+P_2(c)x\bo c)-\sum_{i=1}^q\beta_j (c_i\bo c_i))\qquad (x\in V),
 \]
 Theorem \ref{parts2} implies $e=c$, and hence  $\widetilde\Exp(h)$ and $\widetilde\Exp(h')$ are in the same part in $D(\infty)$ as well, by Proposition \ref{partsbsd2}. Thus, the extension maps each equivalence class in $V(\infty)/\sim$ onto and equivalence class on $D(\infty)/\sim$, as all horofunctions are Busemann.

To complete the proof, we need to show that $\widetilde\Exp$ is injective on $V(\infty)$.  Let $h,h'\in V(\infty)$ be given by
\[
h(x) =  \Lambda_{V_2(e)}(-\frac{1}{2}(e\bo P_2(e)x+P_2(e)x\bo e)-\sum_{i=1}^p\alpha_i (e_i\bo e_i))\qquad (x\in V)
 \]
 and
 \[
h'(x) =  \Lambda_{V_2(c)}(-\frac{1}{2}(c\bo P_2(c)x+P_2(c)x\bo c)-\sum_{i=1}^q\beta_i (c_i\bo c_i))\qquad (x\in V).
 \]
 Set  $g=\widetilde{\Exp}(h)$ and $g' = \widetilde{\Exp}(h')$, and suppose  $g=g'$. Then $e=c$, $p=q$ and
 $\sum_{i=1}^p e^{-\alpha_i}e_i =\sum_{i=1}^q e^{-\beta_i}c_i$ by Proposition \ref{deltabsd}. By Remark \ref{uuni}, we have $\sum_{i=1}^p \alpha_ie_i =\sum_{i=1}^q \beta_ic_i$, and hence $h=h'$ by Corollary \ref{h=h'}, which  concludes the proof.
\end{proof}

\begin{lemma} Let $(v_n)$ be a sequence in $V$ converging to $h\in V(\infty)$. Then $(\widetilde{\Exp}(v_n))$ converges to $\widetilde\Exp(h)$.
\end{lemma}
\begin{proof}
Let $h\in V(\infty)$ with
\[
h(x) =  \Lambda_{V_2(e)}(-\frac{1}{2}(e\bo P_2(e)x+P_2(e)x\bo e)-\sum_{i=1}^p\alpha_i (e_i\bo e_i))\qquad (x\in V),
 \]
 and
  \[
 g(z) =\widetilde\Exp(h) =\frac{1}{2} \log \left\|\sum_{1\leq i\leq j\leq p}e^{-\alpha_i}e^{-\alpha_j}B(z,z)^{-1/2}B(z,e)P_{ij}\right\|
\qquad (z\in D).
 \]
It suffices to show that each subsequence $(\widetilde{\Exp}(v_k))$ has a convergent subsequence with limit $g$.

Since $h\in V(\infty)$, we know $\|v_k\|\to\infty$. Let $r_k =\|v_k\|$ and denote by
\[
v_k = \mu_{1k}c_{1k}+  \mu_{2k}c_{2k} + \cdots +  \mu_{rk}c_{rk}
\]
the spectral decomposition of $v_k$.

Taking subsequences, we may assume that
\begin{enumerate}[(1)]
\item $\beta_{ik} =r_k-\mu_{ik}\to \beta_i \in [0,\infty]$,
\item $c_{ik}\to c_i$,
\end{enumerate}
for all $i=1,\ldots,r$, where $0=\beta_1 \leq \beta_2\leq\ldots\leq \beta_r$.
Let $J=\{i\colon \beta_i<\infty\}=\{1,\ldots, q\}$ and set $c=c_1+\cdots +c_q$.

Observe that
\[
h_{v_k}(x) = \frac{\|x-v_k\|^2-\|v_k\|^2}{\|x-v_k\|+\|v_k\|} =  \frac{(2r_k)^{-1}(\|(x-v_k)\bo (x-v_k)\|-r_k^2)}{2^{-1}(\|r_k^{-1}(x-v_k)\|+1)}
\]
and $2^{-1}(\|r_k^{-1}(x-v_k)\|+1)\to 1$.
Hence  it follows from Lemma \ref{techlem} that $(h_{v_k})$ converges to
\[
h'(x) =  \Lambda_{V_2(c)}(-\frac{1}{2}(c\bo P_2(c)x+P_2(c)x\bo c)-\sum_{i=1}^q\beta_i (c_i\bo c_i))\qquad (x\in V)
\]
and, it follows from  Theorem \ref{formofh} that $h'\in V(\infty)$.

Therefore $h=h'$, and it follows from Corollary \ref{h=h'} that $e =c$, $p=q$ and $\sum_{i=1}^p \alpha_ie_i = \sum_{i=1}^q\beta_ic_i$.  By Remark \ref{uuni}, we get
\[
\sum_{i=1}^p e^{-\alpha_i}e_i = \sum_{i=1}^q e^{-\beta_i}c_i.
\]

Let $w_k = \widetilde\Exp(v_k)$ and $f_{w_k}(z) = \rho(z,w_k) -\rho(0,w_k)$ for $z\in D$. By taking a subsequence, we may assume that $w_k\to\xi\in\partial D$. Note that $\xi$ has a spectral decomposition $\xi = \sum_{i=1}^r \mu_ic_i$, where $\mu_i =1$ for $i=1,\ldots,q$.
Let $g_{-w_k}\colon D \longrightarrow D$ be the M\"obius transformation that maps $w_k$ to $0$. Then
\begin{eqnarray*}
f_{w_k}(z) & = & \rho(z,w_k) - \frac{1}{2}\log \frac{1+\tanh r_k}{1-\tanh r_k}\\
   & = & \frac{1}{2}\log \frac{1+\|g_{-w_k}(z)\|}{1-\|g_{-w_k}(z)\|} - \frac{1}{2}\log \frac{1+\tanh r_k}{1-\tanh r_k}\\
    & = & \frac{1}{2}\log \left(\frac{1-\tanh^2 r_k}{1-\|g_{-w_k}(z)\|^2}\right)\left(\frac{1+\|g_{-w_k}(z)\|}{1+\tanh r_k}\right)^2.
\end{eqnarray*}

By Lemma \ref{hhh},  we have $\lim_k \|g_{-w_k}(z)\|= 1$, so that
 \[
 \left(\frac{1+\|g_{-w_k}(z)\|}{1+\tanh r_k}\right)^2\to 1.
 \]
 As before, we set  $\mu_{0k}=0$ and let $P_{ij}^k$ be the Peirce projections with respect to
 the tripotents $c_{1k},\ldots,c_{rk}$.  Then
 \begin{multline*}
\frac{1-\tanh^2 r_k}{1-\|g_{-w_k}(z)\|^2}  =
\|(1-\tanh^2r_k)B(z,z)^{-1/2}B(z,w^k)B(w_k,w_k)^{-1/2}\|\\
 =   \|\sum_{0\leq i\leq j\leq r} \left(\frac{1-\tanh^2r_k}{1-\tanh^2\mu_{ik}}\right)^{1/2}\left(\frac{1-\tanh^2r_k}{1-\tanh^2\mu_{jk}}\right)^{1/2}B(z,z)^{-1/2}B(z,w_k)P^k_{ij}\|\\
 =   \|\sum_{0\leq i\leq j\leq r} e^{-r_k+\mu_{ik}}e^{-r_k+\mu_{jk}}\left(\frac{1+\tanh r_k}{1+\tanh\mu_{ik}}\right)\left(\frac{1+\tanh r_k}{1+\tanh\mu_{jk}}\right)B(z,z)^{-1/2}B(z,w_k)P^k_{ij}\|\\
 =   \|\sum_{0\leq i\leq j\leq r} e^{-\beta_{ik}}e^{-\beta_{jk}}\left(\frac{1+\tanh r_k}{1+\tanh(r_k-\beta_{ik})}\right)\left(\frac{1+\tanh r_k}{1+\tanh(r_k-\beta_{jk})}\right)B(z,z)^{-1/2}B(z,w_k)P^k_{ij}\|,
 \end{multline*}
 where $\beta_{0k}=r_k$.

For each $i=1,\ldots,q$, we have
 \[
 \frac{1+\tanh r_k}{1+\tanh(r_k-\beta_{ik})}\to 1.
 \]
By \cite[Remark 5.9]{cr}, the Peirce projections $P_{ij}^k$ converge to the Peirce projections $P_{ij}'$ of the minimal tripotents $c_1,\ldots,c_r$, as $c_{ik}\to c_i$ for all $i$.
 Using the fact that $w_k\to \xi = \sum_{i=1}^r \mu_ic_i$ and $\mu_i =1$ for $i=1,\ldots,q$, we find that
  \begin{multline*}
 \|\sum_{0\leq i\leq j\leq r} e^{-\beta_{ik}}e^{-\beta_{jk}}\left(\frac{1+\tanh r_k}{1+\tanh(r_k-\beta_{ik})}\right)\left(\frac{1+\tanh r_k}{1+\tanh(r_k-\beta_{jk})}\right)B(z,z)^{-1/2}B(z,w_k)P^k_{ij}\| \\
  \to  \|\sum_{1\leq i\leq j\leq q} e^{-\beta_{i}}e^{-\beta_{j}}B(z,z)^{-1/2}B(z,\xi)P'_{ij}\|
  =   \|\sum_{1\leq i\leq j\leq q} e^{-\beta_{i}}e^{-\beta_{j}}B(z,z)^{-1/2}B(z,c)P'_{ij}\|
 \end{multline*}
  and hence
 \[
f_{w_k}(z) \to \frac{1}{2}\log \|\sum_{1\leq i\leq j\leq q} e^{-\beta_{i}}e^{-\beta_{j}}B(z,z)^{-1/2}B(z,c)P'_{ij}\|.
\]
The right-hand side is a horofunction, say $f$, in $D(\infty)$ by Theorem \ref{hbsd}. As $e=c$ and  $\sum_{i=1}^p e^{-\alpha_i}e_i = \sum_{i=1}^q e^{-\beta_i}c_i$, we obtain $g=f$ from Proposition \ref{deltabsd}. This shows that
$(\widetilde\Exp(v_k))$ has a subsequence converging to $g$.
\end{proof}

Finally, we prove the following lemma which, together with the preceding ones, complete the proof of Theorem \ref{exp}.
 \begin{lemma} Let $(h_n)$ be a sequence in $V(\infty)$ converging to $h\in V(\infty)$. Then $(\widetilde{\Exp}(h_{n}))$ converges to $\widetilde\Exp(h)$.
\end{lemma}
\begin{proof}
Let $h_n\in V(\infty)$ be given by
\[
h_n(x) =  \Lambda_{V_2(c^n)}(-\frac{1}{2}(c^n\bo P_2(c^n)x+P_2(c^n)x\bo c^n)-\sum_{i=1}^{q_n}\beta_{in} (c_{in}\bo c_{in}))\qquad (x\in V).
 \]
and let $h\in V(\infty)$ be given by
\[
h(x) =  \Lambda_{V_2(e)}(-\frac{1}{2}(e\bo P_2(e)x+P_2(e)x\bo e)-\sum_{i=1}^p\alpha_i (e_i\bo e_i))\qquad (x\in V).
 \]

We show that each subsequence $(\widetilde{\Exp}(h_{k}))$ has a convergent subsequence converging to $\widetilde\Exp(h) =g$, where
  \[
 g(z) =\frac{1}{2} \log \left\|\sum_{1\leq i\leq j\leq p}e^{-\alpha_i}e^{-\alpha_j}B(z,z)^{-1/2}B(z,e)P_{ij}\right\|
\qquad (z\in D).
 \]

Taking  further subsequences, we may assume that $q_k = q_0$ for all $k$, $\beta_{ik}\to \beta_i\in [0,\infty]$ and $c_{ik}\to c_i$ for all $i=1,\ldots,q_0$. As $\min\{\beta_{ik}\colon i=1,\ldots,q_0\} =0$, we have $\min\{\beta_{i}\colon i=1,\ldots,q_0\} =0$. Let  $J=\{i\colon \beta_i<\infty\}$. After relabelling we may assume that $J=\{1,\ldots,q\}$. Let $c =c_1+\cdots+c_q$.  Then Lemma \ref{lem:5.17} implies
 \[
 \lim_k h_k(x)  =  \Lambda_{V_2(c)}(-\frac{1}{2}(c\bo P_2(c)x+P_2(c)x\bo c)-\sum_{i=1}^{q}\beta_{i} (c_{i}\bo c_{i}))\qquad (x\in V).
 \]
 The right-hand side is horofunction by Theorem \ref{formofh}, say $h'\in V(\infty)$. As $h_{k}\to h$, we conclude that $h=h'$ and hence Corollary \ref{h=h'} gives $e =c$, $p=q$, and $\sum_{i=1}^p \alpha_ie_i = \sum_{i=1}^q\beta_ic_i$.  By Remark \ref{uuni}, we have
\[
\sum_{i=1}^p e^{-\alpha_i}e_i = \sum_{i=1}^q e^{-\beta_i}c_i.
\]

 The right-hand side of the following limit is a horofunction $g'$ in $D(\infty)$, by Theorem \ref{hbsd}.
\begin{multline*}
\widetilde\Exp(h_k) =
 \frac{1}{2}\log \|\sum_{1\leq i\leq j\leq q_0} e^{-\beta_{ik}}e^{-\beta_{jk}}B(z,z)^{-1/2}B(z,c^k)P^k_{ij}\| \\
 \to    \frac{1}{2}\log\|\sum_{1\leq i\leq j\leq q} e^{-\beta_{i}}e^{-\beta_{j}}B(z,z)^{-1/2}B(z,c)P'_{ij}\|,
  \end{multline*}
  where  $P_{ij}^k$ are the Peirce projections for the tripotents $c_{1k}, \ldots, c_{q_0k}$, and  $P_{ij}'$  the Peirce projections for the tripotents $c_1, \ldots, c_q$.
  As $e=c$ and  $\sum_{i=1}^p e^{-\alpha_i}e_i = \sum_{i=1}^q e^{-\beta_i}c_i$, Proposition \ref{deltabsd}
   gives  $g=g'$. We conclude that
$(\widetilde\Exp(h_k))$ has a subsequence converging to $g$, which completes the proof.
 \end{proof}

\end{document}